\documentclass[letterpaper,11pt,oneside,reqno]{article}

\usepackage[pdftex,backref=page,colorlinks=true,linkcolor=blue,citecolor=red]{hyperref}
\renewcommand*{\backref}[1]{}
\renewcommand*{\backrefalt}[4]{%
  \ifcase #1%
  \or {\,$\uparrow$#2}%
  \else {\,$\uparrow$#2}%
  \fi}

\usepackage{amsmath,amssymb,amsthm,amsfonts,mathtools}
\usepackage{bm}
\usepackage{graphicx,color}
\usepackage{upgreek}
\usepackage[mathscr]{euscript}

\allowdisplaybreaks
\numberwithin{equation}{section}

\usepackage{tikz}
\usetikzlibrary{shapes,arrows,positioning,decorations.markings}

\usepackage{array}
\usepackage{adjustbox}
\usepackage{cleveref}
\usepackage{aliascnt}
\usepackage[shortlabels]{enumitem}
\setlist[enumerate]{leftmargin=16pt}

\usepackage[DIV=12]{typearea}

\newcommand{\ssp}{\hspace{1pt}}
\DeclareMathOperator{\Det}{det}

\newtheorem{proposition}{Proposition}[section]
\newaliascnt{lemma}{proposition}
\newtheorem{lemma}[lemma]{Lemma}
\aliascntresetthe{lemma}
\crefname{lemma}{Lemma}{Lemmas}\Crefname{lemma}{Lemma}{Lemmas}
\newaliascnt{corollary}{proposition}
\newtheorem{corollary}[corollary]{Corollary}
\aliascntresetthe{corollary}
\crefname{corollary}{Corollary}{Corollaries}\Crefname{corollary}{Corollary}{Corollaries}
\newaliascnt{theorem}{proposition}
\newtheorem{theorem}[theorem]{Theorem}
\aliascntresetthe{theorem}
\crefname{theorem}{Theorem}{Theorems}\Crefname{theorem}{Theorem}{Theorems}

\theoremstyle{definition}
\newaliascnt{conjecture}{proposition}
\newtheorem{conjecture}[conjecture]{Conjecture}
\aliascntresetthe{conjecture}
\crefname{conjecture}{Conjecture}{Conjectures}\Crefname{conjecture}{Conjecture}{Conjectures}
\newaliascnt{definition}{proposition}
\newtheorem{definition}[definition]{Definition}
\aliascntresetthe{definition}
\crefname{definition}{Definition}{Definitions}\Crefname{definition}{Definition}{Definitions}
\newaliascnt{remark}{proposition}
\newtheorem{remark}[remark]{Remark}
\aliascntresetthe{remark}
\crefname{remark}{Remark}{Remarks}\Crefname{remark}{Remark}{Remarks}
\newaliascnt{example}{proposition}

\aliascntresetthe{example}
\crefname{example}{Example}{Examples}\Crefname{example}{Example}{Examples}
\crefname{appendix}{Appendix}{Appendices}\Crefname{appendix}{Appendix}{Appendices}

\begin{document}
\title{A Borodin--Okounkov--Geronimo--Case identity\\for tilted Toeplitz minors}

\author{Leonid Petrov}

\date{}

\maketitle

\begin{abstract}
	We prove a Fredholm determinantal identity for the tilted Toeplitz minor
	\[
	D_{N}^{\xi,\theta}(\varphi)\coloneqq
	\det\bigl[(\theta_{i}\xi_{j}\varphi)_{i-j}\bigr]_{i,j=1}^{N},
	\]
	generalizing the Borodin--Okounkov--Geronimo--Case (BOGC) identity to
	oblique splittings of the Hardy space.
	The tilts $\xi_{j},\theta_{i}$ enter only through an oblique projection
	that multiplies the trace-class kernel $K$ inside the Fredholm
	determinant; the BOGC
	operator $A=I-K$ constructed from $\varphi$ is unchanged.

	Baik--Liao--Liu \cite{BaikLiaoLiu2026} and Liu--Tripathi
	\cite{LiuTripathi2026} have recently shown that the same tilted Toeplitz
	minor admits a contour Fredholm-determinantal representation, in
	connection with the periodic Totally Asymmetric Simple Exclusion Process
	(TASEP). In the periodic TASEP application of Baik--Liao--Liu, the
	formula plays an important role in identifying the periodic KPZ fixed point
	with general initial data.
	Our formula is a companion to their Fredholm determinant and readily reduces to the original BOGC identity.

	The one-sided tilted Toeplitz minor 
	(that is, when all $\theta_i=1$) admits a bialternant form
	recovering Schur and Grothendieck polynomials as special cases.
	A Cauchy--Binet expansion realizes $D_{N}^{\xi,\theta}$ as a
	restricted sum over partitions of products of
	Jacobi--Trudi type determinants, generalizing Gessel's theorem.
	In the pure-shift setting this specializes to a skew Schur expansion.
	Finally, for finite Laurent exponential symbols, we record explicit
	resolvent-block flow identities and formulate the associated
	finite-dimensional closure problem.  
	We also illustrate a possible asymptotic application leading to 
	finite-rank perturbations of the Airy kernel.
\end{abstract}


\section{Introduction}
\label{sec:intro}

\Cref{sub:intro_BOGC} recalls the Borodin--Okounkov--Geronimo--Case (BOGC)
identity \cite{GeronimoCase1979}, \cite{BorodinOkounkov2000},
\cite{BasorWidom2000}, \cite{Bottcher2001}.
\Cref{sub:intro_tilted} introduces the tilted Toeplitz minors and the Hardy
space framework.
The main results are stated in \Cref{sub:intro_main_results}; the relation to
\cite{BaikLiaoLiu2026}, \cite{LiuTripathi2026}, and other works is discussed in
\Cref{sub:intro_related}.

\subsection{Background}
\label{sub:intro_BOGC}

Let $\varphi\colon\mathbb T\to\mathbb C\setminus\{0\}$ be a function on the unit
circle (called the \emph{symbol}).
The $N\times N$ Toeplitz determinant
\begin{equation*}
	D_{N}(\varphi)\coloneqq \det\bigl[\widehat\varphi(i-j)\bigr]_{i,j=0}^{N-1},
	\qquad \widehat\varphi(k)\coloneqq\int_{\mathbb
	T}\varphi(z)\,z^{-k}\,\frac{dz}{2\pi iz},
\end{equation*}
is built from the Fourier coefficients of $\varphi$.
Such determinants have been studied since Szeg\H{o} \cite{Szego1915},
\cite{Szego1952}, mainly for the role of their asymptotic behavior as
$N\to\infty$ in analysis, probability, and mathematical physics, starting with
Onsager's \cite{Onsager1944Ising} solution of the two-dimensional Ising model.

For symbols of zero winding number, the strong Szeg\H{o} limit theorem
\cite{Szego1952} gives the leading-order asymptotic
$D_{N}(\varphi)\sim G(\varphi)^{N}\,Z$ as $N\to\infty$, with $G(\varphi)$ and
$Z$ explicit scalars built from the Fourier coefficients of $\log\varphi$.
The Borodin--Okounkov--Geronimo--Case identity (BOGC, for short) sharpens this
into an exact formula valid at every $N$:
\begin{equation}
	\label{eq:intro_BOGC} D_{N}(\varphi) = G(\varphi)^{N}\cdot Z\cdot
	\det\bigl[\,\delta_{ij}-K_{ij}\,\bigr]_{i,j\ge N}.
\end{equation}
The right-hand side is the Fredholm determinant of the infinite matrix $I-K$
restricted to rows and columns $i,j\ge N$.
The operator $K$ is an $N$-independent product of two Hankel matrices,
\begin{equation*}
	K_{ij}=\bigl(\mathcal H(b)\,\mathcal H(\widetilde c)\bigr)_{ij}, \qquad
	\mathcal H(\psi)_{ij}\coloneqq \widehat\psi(i+j+1),
\end{equation*}
built from the Wiener--Hopf factorization of the symbol into analytic factors
$\varphi=\varphi_{-}\,G(\varphi)\,\varphi_{+}$ (with $\varphi_{-}$ extending
into $|z|>1$ with $\varphi_{-}(\infty)=1$, and $\varphi_{+}$ extending into
$|z|<1$ with $\varphi_{+}(0)=1$) via $b=\varphi_{-}/\varphi_{+}$, $c=1/b$, and
$\widetilde c(z)=c(z^{-1})$.

Identity \eqref{eq:intro_BOGC} originated in Geronimo--Case
\cite{GeronimoCase1979} and was rediscovered in the Toeplitz form by
Borodin--Okounkov \cite{BorodinOkounkov2000}.
It was given two further proofs (including a block-Toeplitz extension) by
Basor--Widom \cite{BasorWidom2000} and B\"ottcher \cite{Bottcher2001}.
The route most directly relevant here is another proof by B\"ottcher--Widom
\cite{BottcherWidom2006}, who recognized \eqref{eq:intro_BOGC} as the
trace-class infinite-dimensional incarnation of the classical \emph{Jacobi
complementary minor identity}: the minor of $A^{-1}$ on a subspace $U$ equals
$\det(A)^{-1}$ times the minor of $A$ on a complementary subspace $V$.
In this way, the BOGC identity corresponds to the special case of the most
natural orthogonal splitting of the Hardy space $H=\ell^{2}(\mathbb Z_{\ge 0})$
(the Hardy space $H^{2}(\mathbb D)$ of analytic functions on the unit disk) into
$U=\mathrm{span}\{e_{0},\dots,e_{N-1}\}$ and $V=\ell^{2}(\mathbb Z_{\ge N})$.
Our starting point is a straightforward extension of this Jacobi identity to
arbitrary (oblique) splittings $H=U\dotplus V$ into closed subspaces.

\subsection{Tilted Toeplitz minors}
\label{sub:intro_tilted}

We now introduce our main object, the tilted Toeplitz minor.
It depends on two collections of analytic functions $\xi_{1},\dots,\xi_{N}$
(extending into $|z|<1$) and $\theta_{1},\dots,\theta_{N}$ (extending into
$|z|>1$), and a symbol $\varphi$, as above.
The \emph{tilted Toeplitz minor} attached to this data is
\begin{equation}
	\label{eq:intro_tilted_def} D_{N}^{\xi,\theta}(\varphi) \coloneqq
	\det\bigl[(\theta_{i}\,\xi_{j}\,\varphi)_{i-j}\bigr]_{i,j=1}^{N}, \qquad
	(f)_{k}\coloneqq \widehat f(k),
\end{equation}
where $\widehat f(k)$ is the $k$-th Fourier coefficient of $f$ on $\mathbb T$.
Assemble $\xi_{j}$ and $\theta_{i}$ into operators $\Xi\colon H\to H$ and
$\Theta\colon H\to\mathbb C^{N}$ on the Hardy space
$H=\ell^{2}(\mathbb Z_{\ge 0})\simeq H^{2}(\mathbb D)$ (the convention of
\cite{BorodinOkounkov2000}, \cite{Bottcher2001}), and form the finite-dimensional pair
\begin{equation}
	\label{eq:intro_chart} R\coloneqq \Theta\,T(\varphi_{+}),\qquad C\coloneqq
	T(\varphi_{-})\,\Xi\,P_{N},
\end{equation}
where $T(\varphi_{\pm})$ are the Toeplitz operators of the Wiener--Hopf factors
which are triangular with respect to the standard basis of $H$, and $P_N$ is the
orthogonal projection onto $\mathrm{span}\{e_{0},\dots,e_{N-1}\}$. By a mild
abuse of notation, we also write $P_{N}$ for the canonical isomorphism
$P_{N}H\simeq\mathbb C^{N}$ where needed, so that maps such as
$T(\varphi_{-})\,\Xi\,P_{N}$ in \eqref{eq:intro_chart} are read as
$\mathbb C^{N}\to H$.
The pair $(R,C)$ generalizes the orthogonal pair $R=C=P_{N}$ of BOGC to an
oblique one; when the ``Gram'' matrix $\Gamma_{\xi,\theta}=RC$ is invertible
(equivalently, $\operatorname{Ran}(C)\cap\operatorname{Ker}(R)=\{0\}$), this
produces the (in general non-orthogonal) splitting
$H=\operatorname{Ran}(C)\dotplus\operatorname{Ker}(R)$.

In our generalizations, the symbol $\varphi$ fixes the operator $A=I-K$.
The tilt $(\xi,\theta)$ is recorded only in the row and column maps $(R,C)$
\eqref{eq:intro_chart}, which define the oblique projection
$\Pi_{V}=I-C(RC)^{-1}R$ onto the space $V=\operatorname{Ker}(R)$.

\subsection{Main results}
\label{sub:intro_main_results}

The main results of the note are the following identities.

\paragraph{Standing assumptions.}
Throughout the rest of the paper,
$\varphi\colon\mathbb T\to\mathbb C\setminus\{0\}$ is of zero winding number
with $\log\varphi\in C^{\alpha}(\mathbb T)$ for some $\alpha>\tfrac{1}{2}$, and
admits the canonical Wiener--Hopf factorization
$\varphi=\varphi_{-}\,G(\varphi)\,\varphi_{+}$ as in \Cref{sub:intro_BOGC}.
Under this assumption $\mathcal H(b)$ and $\mathcal H(\widetilde c)$ are
Hilbert--Schmidt, so $K=\mathcal H(b)\mathcal H(\widetilde c)$ is trace class
and the right-hand side of \eqref{eq:intro_BOGC} is well-defined.
The tilt functions $\xi_{1},\dots,\xi_{N}$ and $\theta_{1},\dots,\theta_{N}$ are
bounded analytic in $|z|<1$ and $|z|>1$ respectively, so that $\Xi$ and $\Theta$
are bounded on $H$.
Rank and transversality of the resulting chart are not automatic for degenerate
tilts; the needed Gram-invertibility hypotheses are stated alongside each
theorem below.

\paragraph{(1) Fredholm determinantal identity for tilted Toeplitz minors (\Cref{thm:tilted_chart}).}
For any tilts $(\xi_{j},\theta_{i})$ subject to the invertibility of the
$N\times N$ Gram matrix
$\Gamma_{\xi,\theta}=\Theta\,T(\varphi_{+})T(\varphi_{-})\,\Xi\,P_{N}$, we have
the identity for the determinant \eqref{eq:intro_tilted_def}:
\begin{equation}
	\label{eq:intro_tilted_main_identity} D_{N}^{\xi,\theta}(\varphi)
	=G(\varphi)^{N}\cdot Z\cdot\Det_{\mathbb C^{N}}\Gamma_{\xi,\theta}\cdot
	\Det_{\operatorname{Ker}(R)}\bigl(I_{\operatorname{Ker}(R)}-\Pi_{V}K|_{\operatorname{Ker}(R)}\bigr).
\end{equation}
The deformed kernel $\Pi_{V}K=K-C\Gamma_{\xi,\theta}^{-1}RK$ differs from $K$ by
a correction of rank at most $N$.
At fixed $N$ this realizes $\Pi_{V}K$ as a finite rank perturbation of the BOGC
operator $K$; in asymptotic regimes where $N$ varies, the identity separates
$D_{N}^{\xi,\theta}(\varphi)$ into the finite-dimensional Gram determinant
$\Det_{\mathbb C^N}\Gamma_{\xi,\theta}$ and the Fredholm determinant on the
oblique tail $\operatorname{Ker}(R)$.

When $RP_N\colon P_NH\to\mathbb C^N$ is also invertible, we identify this
oblique determinant with an ordinary Fredholm determinant on the fixed tail
$Q_NH=\overline{\operatorname{span}}(e_N,e_{N+1},\ldots)$ (see
\Cref{prop:fixed_tail_tilted}); for polynomial tilts of bounded degree, the
resulting kernel is a bounded rank perturbation of $Q_NKQ_N$.  This 
formulation may be useful for soft edge asymptotics, as we discuss in 
\Cref{sec:appendix_spiked}.

\paragraph{(2) Bialternant representation (\Cref{thm:bialternant_column}).}
When the negative part is exactly rank $N$,
$\varphi_{-}(z)=\prod_{l=1}^{N}(1-y_{l}/z)^{-1}$ (with pairwise distinct
$|y_{l}|<1$), and the row tilts are trivial, $\theta_{i}\equiv 1$, the column
tilted minor \eqref{eq:intro_tilted_def} has the bialternant form
\begin{equation*}
 D_{N}^{\xi,\mathbf 1}(\varphi)
	=\det\bigl[(\xi_{j}\varphi)_{i-j}\bigr]_{i,j=1}^{N}
	=G(\varphi)^{N}\cdot\frac{\det\bigl[y_{i}^{N-j}\,\xi_{N-j+1}(y_{i})\bigr]_{i,j=1}^{N}}{\Delta(Y)} \cdot\prod_{l=1}^{N}\varphi_{+}(y_{l}),
\end{equation*}
where $\Delta(Y)=\prod_{1\le i<j\le N}(y_{i}-y_{j})$ is the Vandermonde
determinant.
Particular choices of the analytic functions $\xi_{j}$ recover the Schur and
symmetric Grothendieck polynomials (see the end of
\Cref{sub:bialternant_rational_cauchy} for details).

\paragraph{(3) Two-sided Cauchy--Binet expansion (\Cref{thm:tilted_CB}).}
Under analyticity of $\xi_{j}\varphi_{+}$ and $\theta_{i}\varphi_{-}$ on
overlapping annuli, the tilted Toeplitz minor
\eqref{eq:intro_tilted_def}
admits the absolutely convergent expansion
\begin{equation*}
	D_{N}^{\xi,\theta}(\varphi)
	=G(\varphi)^{N}\sum_{\mu:\,\ell(\mu)\le N}
	\operatorname{JT}_{\mu}^{(N)}(\mathbf a^{\leftarrow})\,
	\operatorname{JT}_{\mu}^{(N)}(\mathbf b^{\leftarrow}),
\end{equation*}
where
$\operatorname{JT}_{\mu}^{(N)}(\mathbf c) \coloneqq \det\bigl[c^{(i)}_{\mu_{j}-j+i}\bigr]_{i,j=1}^{N}$ is a Jacobi--Trudi type determinant, and $\mathbf a^{\leftarrow}$, $\mathbf b^{\leftarrow}$ are the column and row Fourier coefficient sequences of $\xi_{j}\varphi_{+}$ and $\theta_{i}\varphi_{-}$, respectively, indexed in reverse.
This identity generalizes Gessel's theorem \cite{gessel1990symmetric} (recovered at
$\xi=\theta=1$).
The Jacobi--Trudi determinants become skew Schur polynomials
(\Cref{cor:CB_pure_shift_skew_schur}) in a particular case of the pure-shift
tilts $\xi_{j}(z)=z^{a_{j}}$, $\theta_{i}(z)=z^{-b_{i}}$.

\medskip
The main technical input for the proof of \eqref{eq:intro_tilted_main_identity}
is the Grassmannian Jacobi identity (\Cref{thm:GBO}), which extends the
classical finite-dimensional Jacobi complementary minor identity to arbitrary
closed splittings $H=U\dotplus V$ in the trace-class setting.
Its finite rank specialization leads to \eqref{eq:intro_tilted_main_identity}.
The bialternant and Cauchy--Binet identities are derived directly from the
tilted Toeplitz minor.

\subsection{Related work}
\label{sub:intro_related}

Let us discuss the relation of our Fredholm determinantal identity
\eqref{eq:intro_tilted_main_identity} to previous work (the history of the
original BOGC identity was already outlined in \Cref{sub:intro_BOGC}).

\paragraph{(i) Grassmannian language and Sato--Segal--Wilson tau functions.}
There is a well-known Grassmannian interpretation of certain Toeplitz
and block Toeplitz determinants as tau functions. With respect to the Hardy
polarization $L^{2}(\mathbb T)=H_{+}\oplus H_{-}$
(where $H_+$ is the same as our $H$, and $H_-$
is the $L^2$ closure of the span of $z^{-m}$, $m>0$), the Sato--Segal--Wilson
construction \cite{SegalWilson1985},
\cite[Definition~2.6]{CafassoWu2015}
associates to a big-cell point
$W=\operatorname{graph}(h_W)$,
$h_W\colon H_{+}\to H_{-}$, and to a positive loop $g$ the tau function
obtained as follows. Write
\begin{equation*}
	g^{-1}=\begin{pmatrix} d&0\\ b&a\end{pmatrix}
\end{equation*}
for the inverse of $g$. The block matrix is written with rows and
columns ordered as $H_{-}\oplus H_{+}$, so that
$a\colon H_{+}\to H_{+}$, $b\colon H_{-}\to H_{+}$, and
$d\colon H_{-}\to H_{-}$. Then
\begin{equation*}
	\tau_{\mathrm{SSW}}(g;W)
	=\Det_{H_{+}}\bigl(I+a^{-1}b\,h_{W}\bigr).
\end{equation*}
For the loop/Riemann--Hilbert data
considered in \cite{CafassoWu2015},
this tau function is identified with the large-size
Szeg\H{o}--Widom limit $D_{\infty}(g^{-1}\gamma)$ of block Toeplitz
determinants \cite[Theorem~3.4]{CafassoWu2015}. We refer to the latter paper
for further details.

Our use of Grassmannian language is different.
On the open set where $A$ is invertible, and hence $A^{-1}-I$ is trace class
in our setting, choose a unitary identification
$\jmath\colon H_{+}\xrightarrow{\sim}H_{-}$ and associate to $A$
the graph
\begin{equation*}
	W_A=\bigl\{h+\jmath(A^{-1}-I)h\colon h\in H_+\bigr\}.
\end{equation*}
It is a point in the big cell of the restricted Grassmannian associated with the
polarization $H_{+}\oplus H_{-}$, but it is not generally a point of the
SSW/KP Grassmannian, since the invariance condition $zW_A\subset W_A$ need not
hold.

The finite maps $R,C$ (discussed in
\Cref{sub:intro_tilted})
turn the point $W_A$ into a finite exterior matrix
coefficient. Let $c_j=C e_j$ be the columns of $C$, and let $r_i$ be the
coordinate row functionals of $R$. Extend $r_i$ to a functional on
$H_{+}\oplus H_{-}$ by
\begin{equation*}
	\widetilde r_i(u+v)=r_i\bigl(u+\jmath^{-1}v\bigr).
\end{equation*}
Then
\begin{equation*}
	\widetilde r_i\bigl(c_j+\jmath(A^{-1}-I)c_j\bigr)=r_i(A^{-1}c_j),
\end{equation*}
and hence
\begin{equation*}
	\Delta_{R,C}(A)\coloneqq\Det_{\mathbb C^{N}}(RA^{-1}C)
	=\det\bigl[\widetilde r_i\bigl(c_j+\jmath(A^{-1}-I)c_j\bigr)\bigr]_{i,j=1}^{N}.
\end{equation*}
This is the precise sense in which $\Delta_{R,C}(A)$ is a finite coordinate of
$W_A$: for coordinate choices of $R$ and $C$ it is a Pl\"ucker coordinate, while
for general finite maps it is the exterior matrix coefficient obtained by
pairing the Pl\"ucker vector of $W_A$ against
$\widetilde r_1\wedge\cdots\wedge \widetilde r_N$ and
$c_1\wedge\cdots\wedge c_N$.
In \Cref{sec:finite_rank_specialization,sec:tilted_toeplitz_minors}
we prove that
$\Delta_{R,C}(A)=G(\varphi)^{-N}D_{N}^{\xi,\theta}(\varphi)$.
Our identity
\eqref{eq:intro_tilted_main_identity} is the complementary Fredholm form of
this finite coordinate. We emphasize that the tilt is encoded in the finite maps
$R,C$, not in a positive loop $g$ acting on a Grassmannian point as in the SSW
setting.

This coordinate viewpoint is used again in \Cref{sec:dynamics}.
When the symbol $\varphi$ depends on 
times $\mathbf t=(t_1,t_2,\dots)$ via the flow $\varphi(z;\mathbf t)=\exp\bigl(\sum_{k=1}^{M}t_k 
(z^k+z^{-k})\bigr)$,
the operator
$A=A(\mathbf t)=I-K_{\mathbf t}$ and hence the graph point
$W_{A(\mathbf t)}$ move in the restricted Grassmannian. For polynomial tilts,
the corresponding exterior matrix coefficient is represented by the finite
resolvent block
\begin{equation*}
	Y_{\varphi(\cdot;\mathbf t)}^{m,n}=R_m(I-K_{\mathbf t})^{-1}C_n .
\end{equation*}
The identities in \Cref{sec:dynamics} compute the $\mathbf t$-evolution of
these finite matrix coefficients directly.

\paragraph{(ii) Shifted Toeplitz minors and Jacobi--Trudi expansions.}
The particular case of the pure-shift tilt $\xi_{j}(z)=z^{a_{j}}$,
$\theta_{i}(z)=z^{-b_{i}}$ reduces $D_{N}^{\xi,\theta}(\varphi)$ to the shifted
(lacunary) Toeplitz determinant $\det[\widehat\varphi(p_{i}-q_{j})]_{i,j=1}^{N}$
with $p_{i}=i+b_{i}$, $q_{j}=j+a_{j}$.
This shifted determinant was considered by Bump--Diaconis
\cite{BumpDiaconis2002} and Tracy--Widom \cite{TracyWidom2002} in connection
with averages over the unitary group, and further analyzed by Kozlowski
\cite{Kozlowski2013} via Riemann--Hilbert techniques in the regime where the
shifts $a,b$ grow with $N$.
Its closed combinatorial expansion as a finite sum of products of skew Schur
polynomials was given by Maximenko--Moctezuma-Salazar
\cite{MaximenkoMoctezuma2017} and Garc\'\i{}a-Garc\'\i{}a--Tierz
\cite{GarciaTierzGarcia2020}.

\paragraph{(iii) The Liu--Tripathi and Baik--Liao--Liu identities.}
In a recent preprint, Liu--Tripathi \cite[Proposition~1.2]{LiuTripathi2026} prove a
contour Fredholm determinant identity which, under the matching described
below, gives a different formula for the tilted Toeplitz minor:
\begin{equation*}
	D_{N}^{\xi,\theta}(\varphi)
	=\Det_{L^{2}(\Gamma)}\bigl(I+K^{\mathrm{LT}}\bigr),
\end{equation*}
where $K^{\mathrm{LT}}$ is an explicit double contour integral kernel on
$L^{2}(\Gamma)$, of rank at most $N$, built from the same data
$(\xi_{j},\theta_{i},\varphi_{\pm})$.
The earlier identity of Baik--Liao--Liu \cite[Proposition~4.3]{BaikLiaoLiu2026}
is a particular case of \cite{LiuTripathi2026} connected to the periodic TASEP.
In their work, this contour formula plays an important role in identifying the
periodic KPZ fixed point with general initial data.

The starting matrix entries in \cite{LiuTripathi2026} are sums of two contour
integrals,
\begin{equation}
	\label{eq:intro_Liu_Tripathi_entry}
	\oint_{0}v^{i-j-1}\,p_{i}(v)f_{j}(v)\,\frac{dv}{2\pi i}
	+\int_{\Gamma}q_{i}(u)g_{j}(u)\,\frac{du}{2\pi i},
\end{equation}
with $p_{i}$ a polynomial of degree $\le i$, $f_{j}, g_{j}$ analytic functions,
$\Gamma$ a contour enclosing the origin, and the four objects tied by a
reproducing relation $\oint_{0}v^{-i}f_{i}(v)H(v,u)\,dv/(2\pi i)=g_{i}(u)$.
This sum form comes from an \emph{additive} splitting of the symbol
$\varphi=\varphi_{+}+(\varphi-\varphi_{+})$: the inner loop at $0$ captures the
Taylor part $\varphi_{+}$ (analytic in $|z|<1$), and the outer contour $\Gamma$
captures the remainder.
Matching $f_{j}(v)=\widetilde\xi_{j}(v)\,\varphi_{+}(v)$,
$q_{i}(u)=u^{i-1}p_{i}(u)$,
$g_{j}(u)=u^{-j}\widetilde\xi_{j}(u)\,(\varphi-\varphi_{+})(u)$, and
\begin{equation*}
	H(v,u)=\frac{\varphi(u)-\varphi_{+}(u)}{(u-v)\,\varphi_{+}(v)}
\end{equation*}
(with $\widetilde\xi_{j}$ a column tilt and $p_{i}$ a row polynomial of degree
$\le i$), both contours deform to the unit circle, and the entry of
\eqref{eq:intro_Liu_Tripathi_entry} collapses to
$(p_{i}\,\widetilde\xi_{j}\,\varphi)_{j-i}$.
After transposing the matrix and identifying $p_{j}\leftrightarrow\xi_{j}$
(column tilt) and $\widetilde\xi_{i}\leftrightarrow\theta_{i}$ (row tilt), this
is exactly the tilted Toeplitz minor \eqref{eq:intro_tilted_def}.
The contour form of \cite{LiuTripathi2026} is amenable to steepest descent
asymptotic analysis, while ours is more directly connected to the original BOGC
identity.
It would be interesting to understand the precise relation between the two
Fredholm determinants, and to see if one can be transformed into the other
without going through the original determinant \eqref{eq:intro_tilted_def}.

\section*{Acknowledgments}

I thank Zhipeng Liu, whose work inspired this note: the Fredholm determinant
identities developed here arose from looking at Toeplitz minor formulas from
\cite{BaikLiaoLiu2026}, \cite{LiuTripathi2026}, which we discussed at the
workshop \emph{The Kardar--Parisi--Zhang Universality Class \& Related Topics}
at the Brin Mathematics Research Center (BMRC), University of Maryland, College
Park (April 2026).
I am also grateful to Alexei Borodin for helpful discussions.

I was partially supported by the NSF grant DMS-2153869 and by the Simons
Foundation Travel Support for Mathematicians award SFI-MPS-TSM-00013561.

\section{Abstract Grassmannian Jacobi identity}
\label{sec:grassmannian_identity}

Let $H$ be a separable Hilbert space, and let $K$ be a trace class operator
(notation: $K\in\mathfrak S_{1}(H)$).
Assume that
\begin{equation*}
	A\coloneqq I-K
\end{equation*}
is invertible.
Moreover, the Fredholm determinant $\Det_{H}(A)=\Det_H(I-K)$ is well-defined,
and $\Det_{H}(A^{-1})=\Det_H(A)^{-1}$.

Let
\begin{equation}
	\label{eq:UV-decomp} H=U\dotplus V
\end{equation}
be a direct sum decomposition into closed subspaces (not necessarily
orthogonal), and let
\begin{equation*}
	\Pi_{U}\colon H\to U,\qquad \Pi_{V}\colon H\to V
\end{equation*}
be the corresponding projections.
The space of splittings $H=U\dotplus V$ into closed subspaces is the
\emph{Grassmannian} of the section title.
The identity below describes how the observable $\tau_{A}(U,V)$ of
\Cref{def:tau_A} depends on an element of this Grassmannian, for fixed $A$.
We have
\begin{equation*}
	\Pi_{U}+\Pi_{V}=I,\qquad \Pi_{U}\Pi_{V}=\Pi_{V}\Pi_{U}=0,\qquad
	\Pi_{U}^{2}=\Pi_{U},\qquad \Pi_{V}^{2}=\Pi_{V}.
\end{equation*}

\begin{definition}\label{def:tau_A}
	The \emph{Grassmannian observable} attached to $A$ and the Hilbert space
	splitting \eqref{eq:UV-decomp} is defined as
	\begin{equation*}
		\tau_{A}(U,V)\coloneqq \Det_{U}\bigl(\Pi_{U}A^{-1}|_{U}\bigr)
		=\Det_{U}\bigl(I_{U}+\Pi_{U}(A^{-1}-I)|_{U}\bigr).
	\end{equation*}
	The second expression is an identity plus a trace class operator on $U$.
	Indeed, because
	\(A^{-1}-I=A^{-1}K\in\mathfrak S_1(H)\), the compression
	\[
	\Pi_U(A^{-1}-I)|_U:U\to U
	\]
	is trace class.
	Hence \(\tau_A(U,V)\) is well-defined.
\end{definition}

Equivalently, writing \(A^{-1}\) in block form with respect to
\(H=U\dotplus V\),
we have
\begin{equation}
	\label{eq:Ainv_block} A^{-1} =
	\begin{pmatrix}
A^{-1}_{UU} & A^{-1}_{UV}\\[5pt] A^{-1}_{VU} & A^{-1}_{VV}
	\end{pmatrix},
	\qquad A^{-1}_{UU}:U\to U, \qquad \tau_A(U,V)=\Det_U(A^{-1}_{UU}).
\end{equation}
\begin{remark}\label{rem:basis_independence}
	The trace-class condition is invariant under bounded change of basis on $U$
	inherited from a bounded automorphism of $H$ that preserves the splitting
	$H=U\dotplus V$, so $\tau_{A}(U,V)$ depends only on the pair $(U,V)$,
	independently of chosen bases or coordinates within $U$ and $V$.
\end{remark}

The proof of \Cref{thm:GBO} below rests on the following block-triangular
factorization of Fredholm determinants, adapted to oblique splittings.

\begin{lemma}
	\label{lem:block_triangular_det}
	Let $H$ be a separable Hilbert space and $H=U\dotplus V$ a topological
	direct sum of closed subspaces, with associated (oblique) projections
	$\Pi_{U},\Pi_{V}$. Let $T=I+S$ with $S\in\mathfrak S_{1}(H)$, and suppose
	$T$ is block triangular with respect to this splitting, i.e.\ either
	$\Pi_{V}T\Pi_{U}=0$ (upper) or $\Pi_{U}T\Pi_{V}=0$ (lower). Then
	\begin{equation}
		\label{eq:block_triangular_det}
		\Det_{H}(T)=\Det_{U}\bigl(\Pi_{U}T|_{U}\bigr)\,
		            \Det_{V}\bigl(\Pi_{V}T|_{V}\bigr).
	\end{equation}
\end{lemma}
\begin{proof}
	Identify the Banach direct sum $U\oplus V$ with $H$ via the bounded
	bijection $J\colon(u,v)\mapsto u+v$; the inverse
	$J^{-1}h=(\Pi_{U}h,\Pi_{V}h)$ is bounded by the closed graph theorem and
	the topological-sum assumption. Conjugation by $J$ preserves
	$\mathfrak S_{1}$ and Fredholm determinants. After conjugation, the
	diagonal blocks $\Pi_{U}T|_{U}-I_{U}$ and $\Pi_{V}T|_{V}-I_{V}$ and the
	off-diagonal block (either $\Pi_{U}T|_{V}$ in the upper case or
	$\Pi_{V}T|_{U}$ in the lower case) all lie in $\mathfrak S_{1}$.
	Approximate them in $\mathfrak S_{1}$ by finite rank operators of the
	same block form; for the resulting finite rank perturbations of the
	identity, \eqref{eq:block_triangular_det} is the standard linear-algebra
	identity for the determinant of a block-triangular matrix. Pass to the
	limit using continuity of $\Det$ in $\mathfrak S_{1}$
	\cite[Theorem~3.5]{Simon-trace-ideals}.
\end{proof}

\begin{theorem}[Grassmannian Jacobi complementary minor identity]\label{thm:GBO}
	With the notation above, we have
	\begin{equation}
		\label{eq:GBO} \tau_{A}(U,V)
		=\Det_{H}(A^{-1})\Det_{V}\bigl(\Pi_{V}A|_{V}\bigr)
		=\Det_{H}(A^{-1})\Det_{V}\bigl(I_{V}-\Pi_{V}K|_{V}\bigr).
	\end{equation}
\end{theorem}
\begin{remark}
	\label{rmk:finite_dimensional_case} In the finite-dimensional case, identity
	\eqref{eq:GBO} reduces to a classical linear algebra fact known as
	\emph{Jacobi's complementary minor identity}: for an invertible matrix $A$,
	the minor of $A^{-1}$ indexed by $U$ on both sides equals $\Det(A)^{-1}$
	times the minor of $A$ indexed by the complementary subspace $V$.
	\Cref{thm:GBO} extends this to the infinite-dimensional trace class setting.
\end{remark}

\begin{remark}
	In \Cref{sub:reduction_to_BOGC} below we reduce \Cref{thm:GBO} to the
	classical Borodin--Okounkov--Geronimo--Case identity.
\end{remark}

\begin{proof}[Proof of \Cref{thm:GBO}]
	Define
	\begin{equation*}
		B\coloneqq \Pi_{U}+A\Pi_{V}=I-K\Pi_{V},
	\end{equation*}
	where the second equality uses $\Pi_{U}+\Pi_{V}=I$ and $A=I-K$.
	We will compute $\Det_{H}(B)$ and $\Det_{H}(A^{-1}B)$ in two ways and
	compare.
	In the finite-dimensional case this is the standard textbook proof of
	Jacobi's complementary minor identity; we follow that route here, using
	\Cref{lem:block_triangular_det} to handle the block-triangular Fredholm
	determinant factorizations in the trace-class setting.

	We first check that all three operators $A^{-1}$, $B$, and $A^{-1}B$ are of
	the form ``identity plus trace class'', so that their Fredholm determinants
	on $H$ are well-defined.
	For $A^{-1}$, write
	\begin{equation*}
		A^{-1}=I+A^{-1}K,
	\end{equation*}
	and note that $A^{-1}K\in\mathfrak S_{1}(H)$ by the trace-class ideal
	property (product of bounded $A^{-1}$ and trace-class $K$).
	For $B$, the second equality in the definition gives $B-I=-K\Pi_{V}$, again
	trace class.
	Multiplying the two,
	\begin{equation*}
		A^{-1}B =(I+A^{-1}K)(I-K\Pi_{V}) =I+A^{-1}K-K\Pi_{V}-A^{-1}KK\Pi_{V},
	\end{equation*}
	so $A^{-1}B-I$ is a sum of three trace-class terms and is therefore in
	$\mathfrak S_{1}(H)$. (One can equally rewrite this as
	$A^{-1}B-I=A^{-1}K\Pi_{U}$ using $A^{-1}=I+A^{-1}K$ and $\Pi_{U}+\Pi_{V}=I$,
	but only trace-classness matters here.)

	The Fredholm determinant is multiplicative
	\cite[Theorem~3.5]{Simon-trace-ideals}:
	\begin{equation*}
		\Det_{H}\bigl((I+S)(I+T)\bigr) =\Det_{H}(I+S)\Det_{H}(I+T),
	\end{equation*}
	where $S,T\in\mathfrak S_{1}(H)$.
	This implies
	\begin{equation}
		\label{eq:det_mult} \Det_{H}(A^{-1}B)=\Det_{H}(A^{-1})\Det_{H}(B).
	\end{equation}
	Let us identify the factors in \eqref{eq:det_mult} with those in the desired
	identity \eqref{eq:GBO}.
	In the splitting $H=U\dotplus V$, we have the block representations
	\begin{equation*}
		B=
		\begin{pmatrix}I_{U} & \Pi_{U}A|_{V}\\ 0     & \Pi_{V}A|_{V}\end{pmatrix},
		\qquad A^{-1}B=
		\begin{pmatrix}\Pi_{U}A^{-1}|_{U} & 0\\ \Pi_{V}A^{-1}|_{U} & I_{V}\end{pmatrix},
	\end{equation*}
	cf. \eqref{eq:Ainv_block}.
	Both matrices are block-triangular ($B$ upper, $A^{-1}B$ lower), so
	\Cref{lem:block_triangular_det} gives
	\begin{equation}
		\Det_{H}(B)=\Det_{V}\bigl(\Pi_{V}A|_{V}\bigr), \qquad
		\Det_{H}(A^{-1}B)=\Det_{U}\bigl(\Pi_{U}A^{-1}|_{U}\bigr)=\tau_{A}(U,V).
	\end{equation}
	This completes the proof of \Cref{thm:GBO}.
\end{proof}

%
%
%

\section{Specialization to finite rank}
\label{sec:finite_rank_specialization}

Here we specialize \Cref{sec:grassmannian_identity} to the classical setup of
$H=\ell^{2}(\mathbb Z_{\ge 0})$ (the Hardy space of the unit circle) with a
finite-dimensional subspace $U\subset H$.
The latter is encoded by a column map $C\colon\mathbb C^{N}\to H$ with
$\operatorname{Ran}(C)=U$, coupled with a row map $R\colon H\to\mathbb C^{N}$.
In particular, the Grassmannian observable $\tau_{A}(U,V)$ becomes an
$N\times N$ determinant.
The classical Borodin--Okounkov--Geronimo--Case (BOGC) identity for $N\times N$
Toeplitz determinants arises as the orthogonal (``vacuum'') case when $U$ is the
span of the first $N$ basis vectors.

\subsection{Toeplitz operator setup}
\label{sub:hardy_setup}

Take $H=\ell^{2}(\mathbb Z_{\ge 0})$ with the orthonormal basis
$\{e_{k}\}_{k\ge 0}$, identified with the holomorphic basis
$e_{k}\leftrightarrow z^{k}$ on the unit disk
$\mathbb D=\{z\in\mathbb C:|z|<1\}$, with boundary circle $\mathbb T=\{|z|=1\}$.
Let $P_{N}\colon H\to H$ be the orthogonal projection onto the subspace
$\mathrm{span}(e_{0},\ldots,e_{N-1})$, and let $Q_{N}\coloneqq I-P_{N}$ be the
complementary projection.
Set
\begin{equation}
	\label{eq:PN_def} P_{N}H=\operatorname{span}(e_{0},\ldots,e_{N-1}),\qquad
	Q_{N}H=\overline{\operatorname{span}}(e_{N},e_{N+1},\ldots),
\end{equation}
so that $H=P_{N}H\oplus Q_{N}H$ is a closed orthogonal direct sum decomposition.

Equip $\mathbb{T}$ with the normalized Lebesgue measure.
For a function $\varphi\in L^{1}(\mathbb T)$ with Fourier expansion
$\varphi(z)=\sum_{k\in\mathbb Z}\widehat\varphi(k)\,z^{k}$, define the Toeplitz
operator $T(\varphi)\colon H\to H$ and the Hankel operator
$\mathcal H(\varphi)\colon H\to H$ by their matrix entries
\begin{equation*}
	T(\varphi)_{ij}\coloneqq \widehat\varphi(i-j),\qquad \mathcal
	H(\varphi)_{ij}\coloneqq \widehat\varphi(i+j+1),\qquad i,j\ge 0.
\end{equation*}
The associated $N\times N$ \emph{Toeplitz determinant} is
\begin{equation}
	D_{N}(\varphi)\coloneqq \Det_{P_{N}H}\bigl(P_{N}T(\varphi)P_{N}\bigr)
	=\Det\bigl[\widehat\varphi(i-j)\bigr]_{i,j=0}^{N-1}.
\end{equation}
We assume $\varphi\colon\mathbb T\to\mathbb C^{\times}$ is nonvanishing with
zero winding number around the origin, so that $\log\varphi$ is well-defined as
a continuous function on $\mathbb T$.
We also assume that
\begin{equation}
	\label{eq:Holder_assumption} \log\varphi\in C^{\alpha}(\mathbb T)
	\quad\text{for some }\alpha>\tfrac{1}{2}.
\end{equation}
Set
\begin{equation*}
 G(\varphi)\coloneqq \exp\{\widehat{\log\varphi}(0)\},
\end{equation*}
this is often called the geometric mean of~$\varphi$.
Then $\varphi$ admits a
\emph{canonical Wiener--Hopf factorization}
\begin{equation*}
	\varphi(z)=G(\varphi)\varphi_{-}(z)\varphi_{+}(z),
\end{equation*}
where $\varphi_{+}$ extends analytically to $|z|<1$ with $\varphi_{+}(0)=1$, and
$\varphi_{-}$ extends analytically to $|z|>1$ with $\varphi_{-}(\infty)=1$.
Set
\begin{equation*}
	b(z)\coloneqq \frac{\varphi_{-}(z)}{\varphi_{+}(z)},\qquad c(z)\coloneqq
	\frac{1}{b(z)}=\frac{\varphi_{+}(z)}{\varphi_{-}(z)},\qquad \widetilde
	c(z)\coloneqq c(z^{-1}),
\end{equation*}
and define
\begin{equation}
	\label{eq:K_def} K=K^{\varphi}\coloneqq \mathcal H(b)\mathcal H(\widetilde
	c)\colon H\to H.
\end{equation}
Under \eqref{eq:Holder_assumption}, $\mathcal H(b)$ and
$\mathcal H(\widetilde c)$ are Hilbert--Schmidt, so $K\in\mathfrak S_{1}(H)$ is
trace class.
(Indeed, $C^{\alpha}(\mathbb T)\subset H^{s}(\mathbb T)$ for every
$s<\alpha$; choosing $s>\tfrac{1}{2}$ gives
$\sum_{n\ge 1} n\,|\widehat b(n)|^{2}<\infty$ and likewise for
$\widetilde c$, whence the Hankel operators are Hilbert--Schmidt by
the standard criterion; see, e.g., \cite{Peller-Hankel},
\cite{BottcherSilbermann}.)
We use two classical inputs from \cite{BasorWidom2000}, \cite{Bottcher2001},
\cite{Simon-OPUC1}:
\begin{enumerate}
	\item The \emph{operator Wiener--Hopf identity}
	\begin{equation}
		\label{eq:WH_operator_identity}
		T(\varphi)=G(\varphi)T(\varphi_{+})(I-K)^{-1}T(\varphi_{-});
	\end{equation}
	\item The \emph{strong Szeg\H{o} relation}
	\begin{equation}
		\label{eq:strong_Szego} \Det_{H}(I-K)^{-1}=Z,\qquad \log Z=\sum_{k\ge
		1}k\ssp \widehat{\log\varphi}(k)\ssp \widehat{\log\varphi}(-k),
	\end{equation}
	where $\widehat{\log\varphi}(k)$ denotes the $k$-th Fourier coefficient of
	$\log\varphi$.
	Under \eqref{eq:Holder_assumption}, the series in \eqref{eq:strong_Szego}
	converges, and $Z$ is a nonzero complex number.
\end{enumerate}
Throughout the rest of the paper, we denote
\begin{equation}
	\label{eq:A_def} A=A^{\varphi}\coloneqq I-K\colon H\to H,
\end{equation}
so that $Z=\Det_{H}(A^{-1})$.

\subsection{Finite rank chart}
\label{sub:finite_rank_chart}

Fix $N\ge 1$.
A \emph{finite rank chart} on $H$ is a pair of operators
\begin{equation*}
	R\colon H\to\mathbb C^{N}\quad\text{(bounded, surjective)},\qquad
	C\colon\mathbb C^{N}\to H\quad\text{(injective)},
\end{equation*}
together with the \emph{Gram operator}
$\Gamma\coloneqq RC\in\operatorname{End}(\mathbb C^{N})$, which we assume
invertible.
Granted the injectivity of $C$ and surjectivity of $R$, the invertibility of
$\Gamma$ is equivalent to the transversality condition
$\operatorname{Ran}(C)\cap\operatorname{Ker}(R)=\{0\}$.
Set
\begin{equation}
	\label{eq:J_def} J\coloneqq C\Gamma^{-1}\colon \mathbb C^{N}\to H,\qquad
	\Pi_{U}\coloneqq JR=C(RC)^{-1}R,\qquad \Pi_{V}\coloneqq I-\Pi_U.
\end{equation}
both viewed as operators $H\to H$.
Then $RJ=I_{\mathbb C^{N}}$, so $J$ is a bounded right-inverse of~$R$ with
$\operatorname{Ran}(J)=\operatorname{Ran}(C)$; $\Pi_{U}$ projects onto
$\operatorname{Ran}(C)$ along $\operatorname{Ker}(R)$; and $\Pi_{V}$ is the
complementary projection onto $\operatorname{Ker}(R)$ along
$\operatorname{Ran}(C)$.
The splitting
\begin{equation}
	H=U\dotplus V,\qquad U\coloneqq \operatorname{Ran}(C),\qquad V\coloneqq
	\operatorname{Ker}(R),
\end{equation}
satisfies the hypotheses of \Cref{thm:GBO}, with $\dim U=N$.

\subsection{Finite rank Grassmannian Jacobi identity}
\label{sub:finite_GBO}

\begin{theorem}
	\label{thm:finite_GBO} For $K$, $A$ as in \eqref{eq:K_def}, \eqref{eq:A_def}
	depending on a symbol $\varphi$ as in \Cref{sub:hardy_setup}, and $(R,C)$ a
	finite rank chart with $\Gamma=RC$ invertible, we have
	\begin{equation}
		\label{eq:finite_GBO} \Det_{\mathbb C^{N}}\bigl(RA^{-1}C\bigr)
		=Z\cdot\Det_{\mathbb C^{N}}\Gamma\cdot
		\Det_{\operatorname{Ker}(R)}\bigl(I_{\operatorname{Ker}(R)}-\Pi_{V}K|_{\operatorname{Ker}(R)}\bigr),
	\end{equation}
	where $Z$ is the strong Szeg\H{o} normalization \eqref{eq:strong_Szego}.
\end{theorem}

\begin{proof}
	Factor the left-hand side of \eqref{eq:finite_GBO} as
	\begin{equation*}
		RA^{-1}C=(RC)\cdot(RC)^{-1}RA^{-1}C,
	\end{equation*}
	giving
	\begin{equation}
		\Det_{\mathbb C^{N}}\bigl(RA^{-1}C\bigr) =\Det_{\mathbb
		C^{N}}\Gamma\cdot\Det_{\mathbb C^{N}}\bigl(\Gamma^{-1}RA^{-1}C\bigr).
	\end{equation}
	The bijection $C\colon \mathbb C^{N}\to U=\operatorname{Ran}(C)$ has inverse
	$\Gamma^{-1}R|_{U}\colon U\to \mathbb C^{N}$.
	Thus, under the isomorphism $C\colon \mathbb C^{N}\xrightarrow{\sim}U$, the
	operator $(RC)^{-1}RA^{-1}C$ on~$\mathbb C^{N}$ is conjugate to
	\begin{equation*}
		\Pi_{U}A^{-1}\Pi_U=C\Gamma^{-1}RA^{-1}C\Gamma^{-1}R
	\end{equation*}
	on the $N$-dimensional space $U$.
	Thus, their determinants are equal:
	\begin{equation*}
		\Det_{\mathbb C^{N}}\bigl((RC)^{-1}RA^{-1}C\bigr)
		=\Det_{U}\bigl(\Pi_{U}A^{-1}|_{U}\bigr)=\tau_{A}(U,V).
	\end{equation*}
	Applying \Cref{thm:GBO} yields the desired result.
\end{proof}

\begin{lemma}[Finite rank oblique correction]
\label{lem:finite_rank_oblique_correction}
Let $(R,C)$ be a finite rank chart with $\Gamma=RC$ invertible, and let
$f_1,\ldots,f_N$ be the standard basis of $\mathbb C^N$. Set
\begin{equation*}
\qquad c_\alpha\coloneqq C f_\alpha\in H,\qquad
\qquad \psi_\alpha(j)\coloneqq f_\alpha^{\top}\Gamma^{-1}R K e_j,\quad 1\le\alpha\le N.
\end{equation*}
Then, before restriction to $V=\operatorname{Ker}(R)$, we have
\[
(\Pi_VK)(i,j)=K(i,j)-\sum_{\alpha=1}^{N}c_\alpha(i)\,\psi_\alpha(j).
\]
Thus, the oblique kernel is the BOGC kernel plus at most $N$ rank-one
corrections.
\end{lemma}

\begin{proof}
From $\Pi_V=I-C\Gamma^{-1}R$ in \eqref{eq:J_def},
\[
\Pi_VK=K-C\Gamma^{-1}RK
=K-\sum_{\alpha=1}^{N}(Cf_\alpha)\,(f_\alpha^{\top}\Gamma^{-1}RK),
\]
which is the stated coordinate identity.
\end{proof}

\subsection{Reduction to Borodin--Okounkov--Geronimo--Case identity}
\label{sub:reduction_to_BOGC}

The classical Borodin--Okounkov--Geronimo--Case (BOGC) identity
\cite{GeronimoCase1979}, \cite{BorodinOkounkov2000}, \cite{BasorWidom2000},
\cite{Bottcher2001} is the orthogonal (``vacuum'') case $R=C=P_{N}$ of
\Cref{thm:finite_GBO}.
The same Jacobi-based reduction in special case was carried out in
\cite{BottcherWidom2006}.
For completeness, let us record how our general finite rank identity
(\Cref{thm:finite_GBO}) specializes to the BOGC identity:

\begin{corollary}[Borodin--Okounkov--Geronimo--Case identity]\label{cor:BOGC}
	With the notation \eqref{eq:PN_def}--\eqref{eq:strong_Szego}, we have
	\begin{equation}
		D_{N}(\varphi)=G(\varphi)^{N}\cdot Z\cdot
		\Det_{Q_{N}H}\bigl(I_{Q_{N}H}-Q_{N}K|_{Q_{N}H}\bigr).
	\end{equation}
\end{corollary}

\begin{proof}
	Identify $U$ with $\mathbb C^{N}$ in a canonical way, that is, set
	$R=C=P_{N}$ in \Cref{thm:finite_GBO}, where $P_N$ is the orthogonal
	projection onto the first $N$ basis vectors.
	Then $\Gamma$ is the identity on $\mathbb C^{N}$, and
	$\Det_{\mathbb C^{N}}\Gamma=1$.
	We have $U=P_NH$ and $V=Q_NH$, the orthogonal splitting of $H$ along the
	first $N$ basis vectors.
	The right-hand side of \eqref{eq:finite_GBO} becomes
	\begin{equation}
		\label{eq:BOGC_RHS}
		Z\cdot\Det_{Q_{N}H}\bigl(I_{Q_{N}H}-Q_{N}K|_{Q_{N}H}\bigr),
	\end{equation}
	where $Z$ is the strong Szeg\H{o} normalization \eqref{eq:strong_Szego}.
	For the left-hand side, we use the Wiener--Hopf identity
	\eqref{eq:WH_operator_identity} together with the triangular nature of the
	Toeplitz operators $T(\varphi_{\pm})$.
	Namely, since $\varphi_{+}$ is analytic in $|z|<1$ with $\varphi_{+}(0)=1$,
	the matrix $T(\varphi_{+})$ is uni-lower-triangular, hence
	$P_{N}T(\varphi_{+})Q_{N}=0$, so
	\begin{equation}
		\label{eq:Tp_directional} P_{N}T(\varphi_{+})=P_{N}T(\varphi_{+})P_{N}.
	\end{equation}
	Similarly, $T(\varphi_{-})$ is uni-upper-triangular,
	$Q_{N}T(\varphi_{-})P_{N}=0$, so
	\begin{equation}
		\label{eq:Tm_directional} T(\varphi_{-})P_{N}=P_{N}T(\varphi_{-})P_{N}.
	\end{equation}
	Notice that $P_{N}T(\varphi_{\pm})P_{N}$ are unitriangular $N\times N$
	matrices on $P_{N}H$, with determinant~$1$.
	Applying \eqref{eq:Tp_directional}--\eqref{eq:Tm_directional}, we get
	\begin{equation*}
		P_{N}T(\varphi)P_{N}
		=G(\varphi)\bigl(P_{N}T(\varphi_{+})P_{N}\bigr)(I-K)^{-1}
		\bigl(P_{N}T(\varphi_{-})P_{N}\bigr),
	\end{equation*}
	an equality of operators on $H$ that act trivially on $Q_{N}H$, equivalently,
	of operators on $P_{N}H$ after compression. The outer factors are
	unitriangular and act within $P_{N}H$, so restricting to $P_{N}H$ and taking
	determinants yields
	\begin{equation*}
		D_{N}(\varphi) =\Det_{P_{N}H}\bigl(P_{N}T(\varphi)P_{N}\bigr)
		=G(\varphi)^{N}\Det_{P_{N}H}\bigl(P_{N}(I-K)^{-1}P_{N}\bigr)
		=G(\varphi)^{N}\Det_{\mathbb C^{N}}\bigl(RA^{-1}C\bigr),
	\end{equation*}
	where the factor $G(\varphi)^{N}$ comes from the scalar $G(\varphi)$ in
	\eqref{eq:WH_operator_identity} acting on the $N$-dimensional space
	$P_{N}H$.
	Combined with \eqref{eq:finite_GBO} and \eqref{eq:BOGC_RHS}, this gives the
	desired identity.
\end{proof}

\section{Tilted Toeplitz minors}
\label{sec:tilted_toeplitz_minors}

Starting here, we make the operators $R,C$ defining the finite rank chart in
\Cref{thm:finite_GBO} explicit via families of analytic ``tilt'' functions
$\xi_{j}$ and $\theta_{i}$.
The resulting Toeplitz-like determinants are interpreted as Weyl-type
bialternant expressions, and also admit Cauchy--Binet expansions.
The latter extend Gessel's theorem
\cite{gessel1990symmetric}
representing a restricted sum of products of Schur functions as a Toeplitz
determinant.

\subsection{Tilted Toeplitz minors setup}
\label{sub:tilted_toeplitz}

Fix two collections of functions $\xi_{j}$ and $\theta_{i}$, $1\le i,j\le N$:
each $\xi_{j}$ is bounded analytic on $|z|<1$, and each $\theta_{i}$ is bounded
analytic on $|z|>1$ and regular at $z=\infty$.
We impose no normalization at $z=0$ or $z=\infty$, so monomials and arbitrary
analytic profiles are both allowed.
Recall the basis identification $e_{k}\leftrightarrow z^{k}$ for the Hardy space
$H$, under which a vector $h\in H$ is identified with the power series
$h(z)=\sum_{k\ge 0}h_{k}z^{k}$.

\begin{definition}[Tilt operators]
	\label{def:tilt_operators} The \emph{column tilt operator}
	$\Xi\colon H\to H$ acts column by column: it sends the $j$-th basis vector
	$e_{j-1}$ (corresponding to the monomial $z^{j-1}$) to the vector identified
	with the function $z^{j-1}\xi_{j}(z)$, and leaves the basis vectors
	$e_{k}$ for $k\ge N$ untouched.
	In matrix form, the $j$-th column of $\Xi$ for $1\le j\le N$ is the Fourier
	coefficient sequence of $z^{j-1}\xi_{j}(z)$, padded with zeros at indices
	$<j-1$:
	\begin{equation}
		\label{eq:Xi_def} \Xi\,e_{j-1}=\sum_{k\ge
		0}\widehat\xi_{j}(k)\,e_{j-1+k} \quad (1\le j\le N), \qquad
		\Xi\,e_{k}=e_{k}\quad (k\ge N).
	\end{equation}
	This sum is finite whenever each $\xi_{j}$ is a polynomial in $z$.

	The \emph{row tilt operator} $\Theta\colon H\to\mathbb C^{N}$ acts row by
	row: its $i$-th component $(\Theta h)_{i}\in\mathbb C$, for $1\le i\le N$,
	is the $z^{i-1}$-Fourier coefficient of the product $\theta_{i}(z)\,h(z)$.
	Since $\theta_{i}$ is bounded analytic on $|z|>1$ and regular at infinity,
	its Fourier expansion
	$\theta_{i}(z)=\sum_{m\ge 0}\widehat\theta_{i}(-m)\,z^{-m}$ has only
	nonpositive modes, and so
	\begin{equation}
		\label{eq:Theta_def} (\Theta h)_{i}=\sum_{m\ge
		0}\widehat{\theta_{i}}(-m)\,h_{i-1+m} \quad (1\le i\le N).
	\end{equation}
	The sum is finite whenever each $\theta_{i}$ is a polynomial in $z^{-1}$.

	Both $\Xi-I$ and $\Theta$ have rank at most $N$.
\end{definition}

The tilt operators combine with the Wiener--Hopf factors of $\varphi$ into the
finite-dimensional chart maps
\begin{equation}
	\label{eq:tilted_chart} R\coloneqq \Theta\,T(\varphi_{+})\colon H\to\mathbb
	C^{N}, \qquad C\coloneqq T(\varphi_{-})\,\Xi\,P_{N}\colon\mathbb C^{N}\to H.
\end{equation}
The associated Gram operator is
\begin{equation}
	\label{eq:tilted_Gram} \Gamma_{\xi,\theta}\coloneqq R\,C
	=\Theta\,T(\varphi_{+})\,T(\varphi_{-})\,\Xi\,P_{N}\in\operatorname{End}(\mathbb C^{N}).
\end{equation}
When $\Gamma_{\xi,\theta}$ is invertible, these maps give a finite rank chart in
the sense of \Cref{sub:finite_rank_chart}.

\begin{definition}
	\label{def:tilted_minor} The $N\times N$ \emph{tilted Toeplitz minor}
	associated to the symbol $\varphi$ and the tilts $\xi_{j}$, $\theta_{i}$ is
	defined as
	\begin{equation}
		D_{N}^{\xi,\theta}(\varphi) \coloneqq
		\det\Bigl[\bigl(\theta_{i}\,\xi_{j}\,\varphi\bigr)_{i-j}\Bigr]_{i,j=1}^{N},
	\end{equation}
	where $(f)_{k}\coloneqq \widehat f(k)$ denotes the $k$-th Fourier
	coefficient.
\end{definition}

\begin{lemma}\label{lem:tilted_minor_LHS}
	The tilted Toeplitz minor can be written as
	\begin{equation*}
		D_{N}^{\xi,\theta}(\varphi)
		=G(\varphi)^{N}\,\Det_{\mathbb C^{N}}\bigl(R\,A^{-1}\,C\bigr).
	\end{equation*}
\end{lemma}

\begin{proof}
	By the operator Wiener--Hopf identity \eqref{eq:WH_operator_identity},
	\begin{equation*}
		\Theta\,T(\varphi)\,\Xi\,P_{N}
		=G(\varphi)\,\Theta\,T(\varphi_{+})\,(I-K)^{-1}\,T(\varphi_{-})\,\Xi\,P_{N} =G(\varphi)\,R\,A^{-1}\,C,
	\end{equation*}
	where $R$, $C$ are as in \eqref{eq:tilted_chart} and $A=I-K$.

	Let us compute the matrix entries of the operator
	$\Theta\,T(\varphi)\,\Xi\,P_{N}$ on $\mathbb C^{N}$ by tracking Fourier
	coefficients.
	Throughout the calculation, $\varphi$ is a general symbol with Fourier
	coefficients $\widehat\varphi(c)$ defined at all integers $c\in\mathbb Z$;
	note that $\xi_{j}$ and $\theta_{i}$ have constrained Fourier support.

	Fix $1\le j\le N$.
	Since $j-1<N$, the projection $P_{N}$ acts as the identity on $e_{j-1}$, so
	\begin{equation*}
		\Xi\,P_{N}\,e_{j-1}=\sum_{n\ge 0}\widehat\xi_{j}(n)\,e_{j-1+n},
	\end{equation*}
	a vector in $H$ whose $\ell$-th coordinate equals
	$\widehat\xi_{j}(\ell-j+1)$ for $\ell\ge j-1$ and zero otherwise.
	Applying $T(\varphi)$ to the vector above, we obtain a vector whose
	$\ell$-th coordinate is given by
	\begin{equation*}
		\bigl(T(\varphi)\,\Xi\,P_{N}\,e_{j-1}\bigr)_{\ell} =\sum_{k\ge
		0}\widehat\varphi(\ell-k)\, \bigl(\Xi\,P_{N}\,e_{j-1}\bigr)_{k}
		=\sum_{n\ge 0}\widehat\varphi(\ell-j+1-n)\,\widehat\xi_{j}(n).
	\end{equation*}
	Finally, the operator $\Theta$ extracts row Fourier coefficients.
	Substituting the vector above with $\ell=i-1+m$, we get by
	\eqref{eq:Theta_def}:
	\begin{equation}
		\label{eq:tilted_entry_double_sum}
		\bigl(\Theta\,T(\varphi)\,\Xi\,P_{N}\bigr)_{ij} =\sum_{m\ge 0}\sum_{n\ge
		0} \widehat\theta_{i}(-m)\,\widehat\varphi(i-j+m-n)\,\widehat\xi_{j}(n).
	\end{equation}
	We compare \eqref{eq:tilted_entry_double_sum} with the convolution
	\begin{equation*}
		\bigl(\theta_{i}\,\xi_{j}\,\varphi\bigr)_{i-j} =\sum_{a+b+c=i-j}
		\widehat\theta_{i}(a)\,\widehat\xi_{j}(b)\,\widehat\varphi(c),
	\end{equation*}
	where $a,b,c\in\mathbb Z$.
	Note that $\widehat\theta_{i}(a)=0$ for $a>0$ and $\widehat\xi_{j}(b)=0$ for
	$b<0$.
	Setting $a=-m$ ($m\ge 0$), $b=n$ ($n\ge 0$), and $c=i-j+m-n$, we see that
	the convolution above is exactly the double sum in
	\eqref{eq:tilted_entry_double_sum}.
	This completes the proof.
\end{proof}

\begin{theorem}\label{thm:tilted_chart}
	Assume that $\Gamma_{\xi,\theta}$ \eqref{eq:tilted_Gram} is invertible on
	$\mathbb C^{N}$.
	Then the tilted Toeplitz minor admits the Fredholm representation
	\begin{equation}
		\label{eq:tilted_finite_GBO} D_{N}^{\xi,\theta}(\varphi)
		=G(\varphi)^{N}\cdot Z\cdot\Det_{\mathbb
		C^{N}}(\Gamma_{\xi,\theta})\cdot
		\Det_{\operatorname{Ker}(R)}\bigl(I_{\operatorname{Ker}(R)}-\Pi_{V}K|_{\operatorname{Ker}(R)}\bigr),
	\end{equation}
	where $R=\Theta\,T(\varphi_{+})$ and $C=T(\varphi_{-})\,\Xi\,P_{N}$ are the
	chart maps \eqref{eq:tilted_chart}, and
	$\Pi_{V}=I-C\Gamma_{\xi,\theta}^{-1}R$ is the oblique projection onto
	$\operatorname{Ker}(R)$ along $\operatorname{Ran}(C)$.
\end{theorem}

\begin{proof}
	Immediately follows from \Cref{lem:tilted_minor_LHS} and
	\Cref{thm:finite_GBO}.
\end{proof}

\begin{remark}
	Setting $\xi_{j}=\theta_{i}=1$ for all $i,j$
	recovers the BOGC identity (\Cref{cor:BOGC}) from \Cref{thm:tilted_chart}.
	Moreover, the pure shift tilts $\xi_{j}(z)=z^{a_{j}}$,
	$\theta_{i}(z)=z^{-b_{i}}$ for nonnegative integers $a_{j},b_{i}$ recovers
	the shifted Toeplitz minors of \cite{BumpDiaconis2002}: the matrix entries
	in \Cref{def:tilted_minor} reduce to $\widehat\varphi(p_{i}-q_{j})$, where
	$p_{i}\coloneqq i+b_{i}$ and $q_{j}\coloneqq j+a_{j}$.
\end{remark}

\subsection{Properties of the tilted kernel}
\label{sub:properties_of_tilted_kernel}

The kernel $\Pi_{V}K$ in the tilted Fredholm determinant in
\eqref{eq:tilted_finite_GBO} differs from the underlying operator
$K=\mathcal H(b)\mathcal H(\widetilde c)$ \eqref{eq:K_def} of 
the BOGC identity (\Cref{cor:BOGC})
by an
operator of rank at most $N$.
Indeed, since $\Pi_{V}=I-C\Gamma_{\xi,\theta}^{-1}R$, we have
\begin{equation*}
	\Pi_{V}K=K-C\,\Gamma_{\xi,\theta}^{-1}\,R\,K.
\end{equation*}
Since
$C\,\Gamma_{\xi,\theta}^{-1}\,R\,K$
factors through $\mathbb C^{N}$, it has rank at
most $N$.

The kernel $\Pi_VK|_{\operatorname{Ker}(R)}$ 
lives on the oblique tail space $\operatorname{Ker}(R)$,
which depends on the tilts $\xi_j$, $\theta_i$ and, 
as a result, 
may nontrivially depend on $N$.
In the BOGC identity, the tail space
$Q_NH=\overline{\operatorname{span}}(e_N,e_{N+1},\ldots)$ is canonical.
Let us record a version of our tilted identity 
when the oblique tail space is brought back to the canonical form $Q_NH$
by a change of basis.

For shorter notation,
write $P=P_N$, $Q=Q_N$, and set
\begin{equation*}
	B\coloneqq RP|_{PH}\colon PH\longrightarrow \mathbb C^N.
\end{equation*}
When $B$ is invertible, we view $B^{-1}$ as a map $\mathbb C^N\to PH$ and
define the graph parametrization
\begin{equation*}
	\mathcal J_N\colon QH\longrightarrow H,
	\qquad \mathcal J_N y\coloneqq y-B^{-1}Ry,
\end{equation*}
where the second term lies in $PH$.
Equivalently,
\begin{equation*}
	\mathcal J_N=(Q-PB^{-1}RQ)|_{QH}.
\end{equation*}
On $QH$ set
\begin{equation}
	\label{eq:fixed_tail_kernel}
	\mathsf K_N^{\xi,\theta}
	\coloneqq Q\bigl(I-C\Gamma_{\xi,\theta}^{-1}R\bigr)K\mathcal J_N
	\colon QH\to QH,
\end{equation}
and
\begin{equation}
	\label{eq:fixed_tail_correction}
	\mathsf F_N^{\xi,\theta}
	\coloneqq -QKP B^{-1}RQ-QC\Gamma_{\xi,\theta}^{-1}RK\mathcal J_N.
\end{equation}

\begin{proposition}
\label{prop:fixed_tail_tilted} Keep the assumptions of
	\Cref{thm:tilted_chart}, and assume that $B$ is invertible.
	Then $\mathcal J_N$ is a bounded isomorphism from $QH$ onto
	$\operatorname{Ker}(R)$, with inverse $Q|_{\operatorname{Ker}(R)}$.
	Consequently,
	\begin{equation}
		\label{eq:fixed_tail_tilted_GBO} D_N^{\xi,\theta}(\varphi)
		=G(\varphi)^N\cdot Z\cdot\Det_{\mathbb C^N}(\Gamma_{\xi,\theta})
		\cdot\Det_{Q_NH}\bigl(I_{Q_NH}-\mathsf K_N^{\xi,\theta}\bigr).
	\end{equation}
	Moreover, as
	operators $QH\to QH$, we have
	\begin{equation}
		\label{eq:fixed_tail_kernel_expansion}
		\mathsf K_N^{\xi,\theta}=QKQ-QKP B^{-1}RQ
		-QC\Gamma_{\xi,\theta}^{-1}RK\mathcal J_N
		=QKQ+\mathsf F_N^{\xi,\theta}.
	\end{equation}
	In particular, the rank of the correction $\mathsf F_N^{\xi,\theta}$ 
	between the tilted kernel $\mathsf K_N^{\xi,\theta}$ and the BOGC kernel $QKQ$ satisfies
	\begin{equation*}
		\operatorname{rank}\bigl(\mathsf K_N^{\xi,\theta}-QKQ\bigr)
		\le \operatorname{rank}(RQ)+\operatorname{rank}(QC).
	\end{equation*}
\end{proposition}

\begin{proof}
	For $y\in QH$, we have
	\begin{equation*}
		R\mathcal J_Ny=Ry-RB^{-1}Ry=Ry-BB^{-1}Ry=0,
	\end{equation*}
	so $\mathcal J_N(QH)\subseteq\operatorname{Ker}(R)$.
	Conversely, if $v\in\operatorname{Ker}(R)$, then
	$v=Pv+Qv$ and
	\begin{equation*}
		0=Rv=R(Pv)+R(Qv)=B(Pv)+R(Qv).
	\end{equation*}
	Hence $Pv=-B^{-1}R(Qv)$, and therefore
	$v=\mathcal J_N(Qv)$.
	This proves that $\mathcal J_N\colon QH\to\operatorname{Ker}(R)$ is bijective,
	and its inverse is $Q|_{\operatorname{Ker}(R)}$.

	Conjugating the operator
	$I_{\operatorname{Ker}(R)}-\Pi_VK|_{\operatorname{Ker}(R)}$ by this
	isomorphism gives, on $QH$,
	\begin{equation*}
		Q\bigl(I_{\operatorname{Ker}(R)}-\Pi_VK|_{\operatorname{Ker}(R)}\bigr)\mathcal J_N
		=I_{QH}-Q\Pi_VK\mathcal J_N,
	\end{equation*}
	because $Q\mathcal J_N=I_{QH}$.
	The operator $Q\Pi_VK\mathcal J_N$ is trace class on $QH$, since $K$ is
	trace class on $H$ and all other maps in the composition are bounded.
	Since Fredholm determinants are invariant under bounded conjugation,
	\eqref{eq:fixed_tail_tilted_GBO} follows from \Cref{thm:tilted_chart}.

	The expansion \eqref{eq:fixed_tail_kernel_expansion} follows by substituting
	$\Pi_V=I-C\Gamma_{\xi,\theta}^{-1}R$ and
	$\mathcal J_N=Q-PB^{-1}RQ$ into \eqref{eq:fixed_tail_kernel}.
	The first correction term in \eqref{eq:fixed_tail_correction} factors through
	$RQ\colon QH\to\mathbb C^N$, while the second factors through
	$QC\colon\mathbb C^N\to QH$; this yields the rank bound.
\end{proof}

\begin{proposition}
	\label{prop:fixed_tail_bounded_degree}
	Assume the hypotheses of \Cref{prop:fixed_tail_tilted}. If the tilts are
	polynomial of uniformly bounded degrees,
	\begin{equation*}
		\deg \xi_j\le d_\xi,
		\qquad \deg_{z^{-1}}\theta_i\le d_\theta,
	\end{equation*}
	then
	\begin{equation*}
		\operatorname{rank}\bigl(\mathsf K_N^{\xi,\theta}-QKQ\bigr)
		\le d_\xi+d_\theta.
	\end{equation*}
\end{proposition}

For $\xi_j=\theta_i=1$, this correction vanishes because
$d_\xi=d_\theta=0$,
and so
we have
$\mathsf K_N^{\mathbf 1,\mathbf 1}=QKQ$, the original BOGC kernel.

\begin{proof}[Proof of \Cref{prop:fixed_tail_bounded_degree}]
	Because $T(\varphi_+)$ is lower triangular,
	$T(\varphi_+)QH\subseteq QH$.
	If each $\theta_i$ has degree at most $d_\theta$ in $z^{-1}$, then $\Theta$
	restricted to $QH$ only sees the coordinates $e_N,\dots,e_{N+d_\theta-1}$.
	Thus $RQ=\Theta T(\varphi_+)Q$ factors through a space of dimension
	$d_\theta$, so $\operatorname{rank}(RQ)\le d_\theta$.
	Similarly, $\Xi P_N$ has image in
	$\operatorname{span}(e_0,\dots,e_{N+d_\xi-1})$, and the upper triangularity of
	$T(\varphi_-)$ preserves this finite span.
	Therefore, $QC$ has image in
	$\operatorname{span}(e_N,\dots,e_{N+d_\xi-1})$, and
	$\operatorname{rank}(QC)\le d_\xi$.
	Combining these two estimates with the 
	last statement of \Cref{prop:fixed_tail_tilted} yields the desired rank bound.
\end{proof}

\begin{remark}
\label{rmk:column_only_tilt}
	In the column-only case $\theta_i\equiv 1$, identify $\mathbb C^N$ with
	$PH$.  Then $R=PT(\varphi_+)$, and the lower triangularity of
	$T(\varphi_+)$ gives $RQ=PT(\varphi_+)Q=0$.  Since
	$PT(\varphi_+)P|_{PH}$ is triangular with nonzero diagonal, it is invertible;
	hence $\operatorname{Ker}(R)=QH$ and $\mathcal J_N$ is the identity map on
	$QH$.  Thus the fixed-tail conjugation is trivial on the $R$-side, and
	\begin{equation*}
		\mathsf K_N^{\xi,\mathbf 1}
		=QKQ-QC\Gamma_{\xi,\mathbf 1}^{-1}RKQ.
	\end{equation*}
	If only $k$ of the tilts $\xi_j$ differ from $1$,
	then the correction
	$QC=QT(\varphi_-)(\Xi P-P)$ has rank at most $k$.
\end{remark}

\begin{remark}
	\label{rmk:tilted_chart_N1} The preceding formulas take a transparent form
	for a rank-one chart.  Take a column vector $\alpha\in H$ and a row covector
	$\beta^{*}\in H^{*}$, and interpret them as
	$C=\alpha\colon \mathbb C\to H$ and $R=\beta^{*}\colon H\to\mathbb C$.
	Assume that $\beta^{*}\alpha\ne0$, and write
	$\beta^{\perp}\coloneqq\ker\beta^{*}$. Then $\Gamma=\beta^{*}\alpha$, and the
	notation \eqref{eq:J_def} becomes
	\begin{equation*}
		J=\frac{\alpha}{\beta^{*}\alpha}, \qquad
		\Pi_{U}=\frac{\alpha\beta^{*}}{\beta^{*}\alpha}, \qquad
		\Pi_{V}=I-\frac{\alpha\beta^{*}}{\beta^{*}\alpha}.
	\end{equation*}
	The rank-one identity is
	\begin{equation*}
		\beta^{*}A^{-1}\alpha =Z\,\beta^{*}\alpha\cdot
		\Det_{\beta^{\perp}}\bigl(I_{\beta^{\perp}}-
		\Pi_{V}K|_{\beta^{\perp}}\bigr).
	\end{equation*}
	Varying $\alpha$ and $\beta$, one recovers the
	matrix coefficients of the inverse operator $A^{-1}$ through these
	codimension-one Fredholm determinants.

	If, in addition, $\beta^{*}e_0\ne0$, the fixed-tail change of basis above is
	available.  It is
	\begin{equation*}
		\mathcal J_1\colon Q_1H\longrightarrow H,
		\qquad
		\mathcal J_1y=y-\frac{\beta^{*}y}{\beta^{*}e_0}\,e_0,
	\end{equation*}
	and it maps $Q_1H$ isomorphically onto $\beta^{\perp}$, with inverse
	$Q_1|_{\beta^{\perp}}$.  The fixed-tail kernel is therefore
	\begin{equation*}
		\mathsf K_1^{\alpha,\beta}
		\coloneqq Q_1\left(I-\frac{\alpha\beta^{*}}{\beta^{*}\alpha}\right)
		K\mathcal J_1\colon Q_1H\to Q_1H,
	\end{equation*}
	and the determinant becomes
	\begin{equation*}
		\beta^{*}A^{-1}\alpha
		=Z\,\beta^{*}\alpha\cdot
		\Det_{Q_1H}\bigl(I_{Q_1H}-\mathsf K_1^{\alpha,\beta}\bigr).
	\end{equation*}
	In this form the finite-rank correction is explicit:
	\begin{equation*}
		\mathsf K_1^{\alpha,\beta}=Q_1KQ_1
		-\frac{Q_1Ke_0\,\beta^{*}|_{Q_1H}}{\beta^{*}e_0}
		-\frac{Q_1\alpha\,\beta^{*}K\mathcal J_1}{\beta^{*}\alpha}.
	\end{equation*}
	Thus, the difference from the ordinary tail kernel $Q_1KQ_1$ is the sum of at
	most two rank-one operators on $Q_1H$.
\end{remark}

\subsection{Bialternant form of one-sided tilted Toeplitz minors}
\label{sub:bialternant_rational_cauchy}

Throughout this subsection we write $D_{N}^{\xi,\mathbf 1}$ and
$D_{N}^{\mathbf 1,\theta}$ for the column-only and row-only specializations of
the tilted Toeplitz minor of \Cref{def:tilted_minor}, and we assume
$G(\varphi)=1$ in each case.

\begin{definition}[Rank-$N$ rational Wiener--Hopf factors]
	\label{def:rank_N_rational_minus} \label{def:rank_N_rational_plus} Fix
	tuples $X=(x_{1},\dots,x_{N})$ and $Y=(y_{1},\dots,y_{N})$ of pairwise
	distinct complex numbers with $|x_{k}|<1$ and $|y_{l}|<1$.
	We say a symbol $\varphi$ has rank-$N$ rational factor $\varphi_{+}$ or
	$\varphi_-$ if, respectively,
	\begin{equation*}
		\varphi_{+}(z)=\prod_{k=1}^{N}\frac{1}{1-x_{k}z}, \quad \text{or}\quad
		\varphi_{-}(z)=\prod_{l=1}^{N}\frac{1}{1-y_{l}/z}.
	\end{equation*}
	When one of these rational conditions is imposed, the other factor is
	arbitrary, subject to the usual assumptions of
	\Cref{sub:hardy_setup} on the full symbol $\varphi$.
\end{definition}

\begin{theorem}[Bialternant factorization, column-side]
	\label{thm:bialternant_column} Let $\varphi$ have rank-$N$ rational
	$\varphi_{-}$ with alphabet $Y$ as in \Cref{def:rank_N_rational_minus}.
	Let $\xi_{1},\dots,\xi_{N}$ be analytic in a neighborhood of $|z|<1$.
	Then the column-only tilted Toeplitz minor factorizes as
	\begin{equation}
		\label{eq:bialternant_column} D_{N}^{\xi,\mathbf 1}(\varphi) =\mathcal
		S_{\xi}(Y)\cdot\prod_{l=1}^{N}\varphi_{+}(y_{l}), \qquad \mathcal
		S_{\xi}(Y)\coloneqq
		\frac{\det\bigl[y_{i}^{N-j}\,\xi_{N-j+1}(y_{i})\bigr]_{i,j=1}^{N}}{\Delta(Y)},
	\end{equation}
	where $\Delta(Y)=\det[y_{i}^{N-j}]_{i,j=1}^{N}=\prod_{i<j}(y_{i}-y_{j})$ is
	the Vandermonde determinant of $Y$.
\end{theorem}

\begin{proof}
	Set $P(z)\coloneqq\prod_{\ell=1}^{N}(z-y_{\ell})$.
	Since
	\begin{equation*}
		\varphi_{-}(z)=\prod_{\ell=1}^{N}\frac{1}{1-y_{\ell}/z}=\frac{z^{N}}{P(z)},
	\end{equation*}
	we have $\varphi(z)=\varphi_{+}(z)\,z^{N}/P(z)$ under the normalization
	$G(\varphi)=1$.
	The $(i,j)$-entry of the tilted Toeplitz minor is the Fourier coefficient
	\begin{equation*}
		M_{ij}=\frac{1}{2\pi
		i}\oint_{|z|=\rho}\xi_{j}(z)\varphi(z)\,z^{j-i-1}\,dz =\frac{1}{2\pi
		i}\oint_{|z|=\rho}\xi_{j}(z)\varphi_{+}(z)\,
		\frac{z^{N+j-i-1}}{P(z)}\,dz,
	\end{equation*}
	where we choose $\max_{\ell}|y_{\ell}|<\rho<1$, so $\xi_{j}\varphi_{+}$ is
	analytic inside the contour.
	Since $1\le i,j\le N$, the exponent $N+j-i-1\ge 0$, so the integrand has no
	pole at $z=0$; the only poles inside $|z|=\rho$ are the simple poles at
	$z=y_{\ell}$.
	By the residue theorem,
	\begin{equation*}
		M_{ij}=\sum_{\ell=1}^{N}\xi_{j}(y_{\ell})\,\varphi_{+}(y_{\ell})\,
		\frac{y_{\ell}^{N+j-i-1}}{P'(y_{\ell})}
		=\sum_{\ell=1}^{N}y_{\ell}^{N-i}\cdot
		\frac{\varphi_{+}(y_{\ell})}{P'(y_{\ell})}\cdot
		y_{\ell}^{j-1}\xi_{j}(y_{\ell}).
	\end{equation*}
	This implies the decomposition $M=ADB_{\xi}$ with
	$A_{i\ell}=y_{\ell}^{N-i}$,
	$D=\operatorname{diag}\bigl(\varphi_{+}(y_{\ell})/P'(y_{\ell})\bigr)$,
	$(B_{\xi})_{\ell j}=y_{\ell}^{j-1}\xi_{j}(y_{\ell})$.
	The first factor gives $\Delta(Y)$ in the numerator, but
	$\prod_{\ell}P'(y_{\ell}) =(-1)^{N(N-1)/2}\Delta(Y)^{2}$, which leaves a
	single Vandermonde in the denominator.
	The determinant of $B_\xi$ becomes, up to the sign $(-1)^{N(N-1)/2}$ from
	the denominator, the numerator in the bialternant $\mathcal S_{\xi}(Y)$.
	This completes the proof.
\end{proof}

\begin{corollary}
	\label{cor:DN_rank_N_rational_minus} Specializing $\xi_{j}\equiv 1$ in
	\Cref{thm:bialternant_column} gives unshifted Toeplitz minor
	$D_{N}(\varphi)=\prod_{l=1}^{N}\varphi_{+}(y_{l})$, so that
	\eqref{eq:bialternant_column} can be rewritten as
	$D_{N}^{\xi,\mathbf 1}(\varphi)=\mathcal S_{\xi}(Y)\cdot D_{N}(\varphi)$.
\end{corollary}

\begin{remark}
	\label{rmk:bialternant_row} A similar bialternant factorization holds for
	the row-only tilted Toeplitz minor $D_{N}^{\mathbf 1,\theta}(\varphi)$ when
	$\varphi$ has rank-$N$ rational $\varphi_{+}$ with alphabet
	$X=(x_{1},\dots,x_{N})$.
	We omit the statement here, but only record the corresponding bialternant
	\begin{equation*}
		\mathcal S_{\theta}^{\#}(X)\coloneqq
		\frac{\det\bigl[x_{k}^{-(N-i)}\,\theta_{N-i+1}(1/x_{k})\bigr]_{i,k=1}^{N}}{\Delta(X^{-1})}.
	\end{equation*}
\end{remark}

\begin{remark}[Smaller alphabet on the residue side]
	\label{rmk:smaller_alphabet} The hypotheses of \Cref{thm:bialternant_column}
	require the alphabet $Y$ to have exactly $N$ entries.
	If $\varphi_{-}$ has $|Y|<N$ poles, formula \eqref{eq:bialternant_column}
	can be specialized to setting the extra entries of $Y$ to zero and using the
	l'H\^opital rule to cancel the resulting zeros in the numerator and
	denominator of $\mathcal S_{\xi}(Y)$.
\end{remark}

Let us record a special case of the bialternant expression
$\mathcal{S}_{\xi}(Y)$ arising in \Cref{thm:bialternant_column}.
First, if $\lambda=(\lambda_1\ge\cdots\ge\lambda_N\ge 0)$ is a partition with at
most $N$ parts, $\widehat\lambda_{j}=\lambda_{N+1-j}$, and the tilt functions
are binomials $\xi_{j}(z)=z^{\widehat\lambda_{j}}(1+\beta z)^{N-j}$, then
\begin{equation*}
	\mathcal S_{\xi}(Y) =\frac{\det\bigl[y_{i}^{\lambda_j+N-j}(1+\beta
	y_{i})^{j-1}\bigr]_{i,j=1}^{N}}{\Delta(Y)} \eqqcolon
	G_{\lambda}^{\beta}(y_1,\ldots,y_N ),
\end{equation*}
which is the symmetric Grothendieck polynomial,
\cite{lascoux1982structure},\cite{fomin1994grothendieck},\cite{buch2002littlewood},\cite{yeliussizov2015duality}.
Further taking $\beta=0$ recovers the Schur polynomial
$s_{\lambda}(y_1,\ldots,y_N )$.

An alternative choice
\begin{equation*}
	\xi_{j}(z)=(1+\beta z)^{\widehat\lambda_{j}}
\end{equation*}
yields a different Grothendieck-type bialternant which we denote as
\begin{equation*}
	\mathcal S_{\xi}(Y)
	=\frac{\det\bigl[y_{i}^{N-j}(1+\beta y_{i})^{\lambda_{j}}\bigr]_{i,j=1}^{N}}
	{\Delta(Y)} \eqqcolon \tilde G_{\lambda}^{\beta}(y_1,\ldots,y_N ).
\end{equation*}
This determinant essentially appears in \cite[(3.16)]{BaikLiaoLiu2026}, and is
related to the ``dual'' Grothendieck-type polynomials from
\cite{Motegi-Sakai13}.
Let us emphasize that the ``dual'' term used in \cite{Motegi-Sakai13} conflicts
with the more standard dual Grothendieck polynomials $g_\lambda$ arising from
the original ones through the Hall inner product on the algebra of symmetric
functions \cite{Macdonald1995}.
We refer to \cite{yeliussizov2015duality}, \cite{hwang2021refined} for details
on the $g_\lambda$'s and their combinatorial interpretations.

\subsection{Cauchy--Binet expansion}
\label{sub:cauchy_binet}

Here we expand the two-sided tilted Toeplitz minor $D_{N}^{\xi,\theta}(\varphi)$
into a restricted sum over partitions, generalizing Gessel's theorem
\cite{gessel1990symmetric} that was a key ingredient in the original proof of
the BOGC identity in \cite{BorodinOkounkov2000}.

Throughout this subsection we assume there exists $R>1$ such that
$\xi_{j}\varphi_{+}$ is holomorphic in $|z|<R$ for each $j$, and
$\theta_{i}\varphi_{-}$ is holomorphic in $|z|>R^{-1}$ and regular at $\infty$
for each $i$.
Expand
\begin{equation*}
 \xi_{j}(z)\,\varphi_{+}(z)=\sum_{r\ge
	0}a_{r}^{(j)}\,z^{r} \qquad \theta_{i}(z)\,\varphi_{-}(z)=\sum_{r\ge
	0}b_{r}^{(i)}\,z^{-r},
\end{equation*}
where $1\le i,j\le N$.
Set $a_{r}^{(j)}\coloneqq 0$, $b_{r}^{(i)}\coloneqq 0$ for $r<0$.
Denote the coefficient sequences by $\mathbf{a}^{(j)}=(a_{r}^{(j)})_{r\ge 0}$
and $\mathbf{b}^{(i)}=(b_{r}^{(i)})_{r\ge 0}$, and write
\begin{equation*}
	\mathbf a=(\mathbf a^{(1)},\dots,\mathbf a^{(N)}),\quad \mathbf b=(\mathbf
	b^{(1)},\dots,\mathbf b^{(N)}), \quad \mathbf a^{\leftarrow}=(\mathbf
	a^{(N)},\dots,\mathbf a^{(1)}), \quad \mathbf b^{\leftarrow}=(\mathbf
	b^{(N)},\dots,\mathbf b^{(1)}).
\end{equation*}

\begin{definition}
	\label{def:tilted_JT} Let
	$\mathbf c=(\mathbf c^{(1)},\dots,\mathbf c^{(N)})$ be a family of one-sided
	sequences $\mathbf c^{(i)}=(c_{r}^{(i)})_{r\ge 0}$ (with
	$c_{r}^{(i)}\coloneqq 0$ for $r<0$), and let
	$\mu=(\mu_{1}\ge\cdots\ge\mu_{N}\ge 0)$ be a partition with at most $N$
	parts.
	The associated \emph{Jacobi--Trudi type determinant} is
	\begin{equation*}
 \operatorname{JT}_{\mu}^{(N)}(\mathbf c)
		\coloneqq \det\bigl[c_{\mu_{j}-j+i}^{(i)}\bigr]_{i,j=1}^{N}.
	\end{equation*}
\end{definition}

\begin{theorem}[Cauchy--Binet expansion]
	\label{thm:tilted_CB} Under the above assumptions, we have the absolutely
	convergent identity
	\begin{equation}
		\label{eq:tilted_CB_partition} D_{N}^{\xi,\theta}(\varphi)
		=G(\varphi)^{N}\sum_{\mu\colon\ell(\mu)\le N}
		\operatorname{JT}_{\mu}^{(N)}(\mathbf a^{\leftarrow})\,
		\operatorname{JT}_{\mu}^{(N)}(\mathbf b^{\leftarrow}).
	\end{equation}
\end{theorem}

\begin{proof}
	Set $A_{j}(z)\coloneqq\xi_{j}(z)\varphi_{+}(z)$ and
	$B_{i}(z)\coloneqq\theta_{i}(z)\varphi_{-}(z)$.
	We have
	\begin{equation*}
		[z^{i-j}]\,B_{i}(z)A_{j}(z) =\sum_{s\ge 0}b_{s}^{(i)}\,a_{s+i-j}^{(j)}
		=\sum_{k\ge 0}b_{k-i+1}^{(i)}\,a_{k-j+1}^{(j)} \qquad(k\coloneqq s+i-1),
	\end{equation*}
	using $a_{r}^{(j)}=b_{r}^{(i)}=0$ for $r<0$.
	Set $\mathsf B_{i,k}\coloneqq b_{k-i+1}^{(i)}$,
	$\mathsf A_{k,j}\coloneqq a_{k-j+1}^{(j)}$ ($k\ge 0$).
	Then the matrix of the tilted minor $D_{N}^{\xi,\theta}(\varphi)$ factors as
	$G(\varphi)\,\mathsf B\mathsf A$ with $\mathsf B$ of size $N\times\infty$
	and $\mathsf A$ of size $\infty\times N$.

	Cauchy--Binet applied to finite truncations $\mathsf B^{(M)}\mathsf A^{(M)}$
	in the $k$-index gives
	\begin{equation*}
		\det\bigl(\mathsf B^{(M)}\mathsf A^{(M)}\bigr) =\sum_{0\le
		k_{1}<\cdots<k_{N}\le M}
		\det\bigl[b_{k_{\ell}-i+1}^{(i)}\bigr]_{i,\ell=1}^{N}
		\det\bigl[a_{k_{\ell}-j+1}^{(j)}\bigr]_{\ell,j=1}^{N}.
	\end{equation*}
	Fix any $\rho\in(1,R)$.
	Cauchy's estimate on $|z|=\rho$ for $\xi_{j}\varphi_{+}$ and on
	$|z|=\rho^{-1}$ for $\theta_{i}\varphi_{-}$ (holomorphic at $\infty$ via
	$w=z^{-1}$) yields $|a_{r}^{(j)}|+|b_{r}^{(i)}|\le C\rho^{-r}$ for all
	$r\ge 0$, uniformly in $1\le i,j\le N$.
	Hence each $N\times N$ minor in the sum is bounded by
	$N!\,C^{N}\rho^{N(N-1)/2}\rho^{-(k_{1}+\cdots+k_{N})}$, and the product of
	the two minors on the right is $O(\rho^{-2(k_{1}+\cdots+k_{N})})$.
	Thus, the limit $M\to\infty$ exists and yields, after multiplication by
	$G(\varphi)^{N}$, the absolutely convergent identity
	\begin{equation*}
		D_{N}^{\xi,\theta}(\varphi) =G(\varphi)^{N} \sum_{0\le
		k_{1}<\cdots<k_{N}}
		\det\bigl[b_{k_{\ell}-i+1}^{(i)}\bigr]_{i,\ell=1}^{N}
		\det\bigl[a_{k_{\ell}-j+1}^{(j)}\bigr]_{\ell,j=1}^{N}.
	\end{equation*}
	Reparameterizing $k_{\ell}=\mu_{N-\ell+1}+\ell-1$ yields partitions $\mu$
	with at most $N$ parts, and the minors in the right-hand side above become
	the Jacobi--Trudi determinants
	$\operatorname{JT}_{\mu}^{(N)}(\mathbf b^{\leftarrow})$ and
	$\operatorname{JT}_{\mu}^{(N)}(\mathbf a^{\leftarrow})$.
	This completes the proof.
\end{proof}

\begin{remark}
\label{rmk:tilted_CB_biorthogonal_ensemble}
	The sum \eqref{eq:tilted_CB_partition} is a partition function of a
	$N$-point \emph{biorthogonal ensemble} in the sense of
	\cite{Borodin1998b}, since the Jacobi--Trudi determinants
	have the form
	\begin{equation*}
		\operatorname{JT}_{\mu}^{(N)}(\mathbf c) =
		\det[f_{i}(m_j)]_{i,j=1}^{N}, \qquad f_{i}(r)=c_{r+i-N}^{(i)},
		\qquad m_j=\mu_j+N-j,\ j=1,\ldots,N.
	\end{equation*}
\end{remark}

\begin{remark}
\label{rmk:not_GP_grothendieck}
	The biorthogonal weight in \eqref{eq:tilted_CB_partition} should be
	distinguished from the Grothendieck measure on partitions
	(and from the tilted biorthogonal ensemble) introduced in
	\cite{GavrilovaPetrov2023_Groth}.
	In that construction, each particle $j$ carries a finite-difference
	operator chain of length $j-1$, column-flagged by the particle position.
	Here the row tilt $\theta_{i}$ is absorbed into the row-indexed sequence
	$\mathbf b^{(i)}$ inside the Jacobi--Trudi determinant.
	The two ensembles share the feature that their probability weights are
	given by products of two determinants, but the tilt is placed in different
	data.
	It would be interesting to find a common generalization of the two objects.
\end{remark}

\subsection{Schur symbols}
\label{sub:schur_symbols}

Let us consider the case when the symbol $\varphi(z)$ 
corresponds to two generic Schur-positive 
specializations
$\rho^{\pm}=(\alpha^{\pm};\beta^{\pm};\gamma^{\pm})$
of the algebra of symmetric functions associated with the factors $\varphi_{+}$ and $\varphi_{-}$.
We refer to \cite{borodin2016representations} for background on Schur-positive specializations
and their connection to Toeplitz determinants and representations of the infinite symmetric group.
We set
\begin{equation*}
	\varphi(z)=e^{\gamma^+ z}\prod_{k=1}^{\infty}\frac{1+\beta_k^+ z}{1-\alpha_k^+ z}\cdot
	e^{\gamma^- z^{-1}}\prod_{l=1}^{\infty}\frac{1+\beta_l^- z^{-1}}{1-\alpha_l^- z^{-1}},
\end{equation*}
where
\begin{equation*}
	\sum_{k=1}^{\infty}(\alpha_k^+ +\beta_k^+) <\infty, \qquad
	\sum_{l=1}^{\infty}(\alpha_l^- +\beta_l^-) <\infty, \qquad
	\gamma^{+},\gamma^{-}\ge 0.
\end{equation*}

In the case of the trivial tilt $\xi_{j}=\theta_{i}=1$ for all $i,j$,
\Cref{thm:tilted_CB} reduces to Gessel's theorem \cite{gessel1990symmetric}.
Indeed, then the coefficients $a^{(i)}_{r},b^{(i)}_{r}$ 
are independent of $i$, and equal to the complete homogeneous symmetric functions
under the specializations $\rho^{\pm}$. 
\begin{corollary}[Gessel's theorem \cite{gessel1990symmetric}]
	\label{cor:CB_unshifted} With the above assumptions, we have the following
	identity for the non-tilted Toeplitz determinants:
	\begin{equation}
	\label{eq:CB_unshifted}
		D_{N}(\varphi)=\sum_{\mu\colon\ell(\mu)\le N} s_{\mu}(\rho^{+})\,s_{\mu}(\rho^{-}).
	\end{equation}
\end{corollary}

\begin{remark}
	When one of the specializations is a pure alpha specialization with
	at most $N$ nonzero $\alpha$-parameters, the restriction $\ell(\mu)\le N$ in the sum above is automatically satisfied.
	Therefore, in this case, $D_N(\varphi)$ has a product form thanks to the Cauchy summation
	identity for Schur functions \cite[I.(4.3)]{Macdonald1995}. If $\rho^-$ is such a finite alpha specialization,
	the same product form follows by setting $\xi_j\equiv 1$ in
	\Cref{thm:bialternant_column}; if $\rho^+$ is finite alpha, it follows from
	the analogous row-side bialternant of \Cref{rmk:bialternant_row}.

	In the general case, the biorthogonal ensemble arising
	from the right-hand side of \eqref{eq:CB_unshifted} is a \emph{conditional Schur measure},
	conditioned on the length of the partition being at most $N$.
	Asymptotic analysis of such conditional ensembles is delicate and usually
	calls for methods other than computing the correlation kernel in the
	conditional setting.
\end{remark}

In the case of two-sided pure shift tilts
$\xi_j(z)=z^{\lambda_{N+1-j}}$, $\theta_i(z)=z^{-\nu_{N+1-i}}$ for partitions
$\lambda,\nu$ with at most $N$ parts, the Jacobi--Trudi determinants in
\Cref{thm:tilted_CB} reduce to skew Schur polynomials under the specializations
$\rho^{\pm}$.
\begin{corollary}[Skew Schur expansion]
	\label{cor:CB_pure_shift_skew_schur} With the above assumptions and
	notation, 
	we have
	\begin{equation}
		\label{eq:CB_pure_shift_skew_schur}
		D_{N}^{\xi,\theta}(\varphi)
		=\sum_{\eta\colon\ell(\eta)\le N} s_{\eta/\lambda}(\rho^+)\,s_{\eta/\nu}(\rho^-).
	\end{equation}
\end{corollary}

Under Schur-positive specializations
$\rho^{\pm}$, the
summands in \eqref{eq:CB_pure_shift_skew_schur} are nonnegative.
Thus, $D_{N}^{\xi,\theta}(\varphi)$ is the normalizing constant for a
probability measure on partitions with at most $N$ parts, supported on 
$\eta$'s with 
$\eta\supseteq\lambda,\nu$.

\section{Resolvent flows and a finite-dimensional closure problem}
\label{sec:dynamics}

Specialize the symbol $\varphi(z)$
to the family depending on finitely many times $\mathbf t=(t_{1},\ldots,t_{M})$ given by
\begin{equation}
	\label{eq:Laurent_symbol_finite_times}
	\varphi(z;\mathbf t)\coloneqq \exp\Bigl(\sum_{r=1}^{M}t_{r}(z^{r}+z^{-r})\Bigr).
\end{equation}
In the case $M=1$, this is the Bessel symbol $\varphi_{t}(z)=e^{t(z+z^{-1})}$.
The goal of this section is to describe the dependence of the tilted Toeplitz minors on the times $\mathbf t$,
inspired 
by the
differential equation method for
integrable Fredholm determinants of Its--Izergin--Korepin--Slavnov
\cite{its1990differential}; the Airy-kernel $F_2$/Painlev\'e~II formula of
Tracy--Widom \cite{TW_Airy_Painleve_2002} is the model example. For the Bessel
Toeplitz determinants, related Painlev\'e equations and recurrences for
orthogonal polynomials on the unit circle (OPUC) are part of the classical scalar
theory, recalled below in \Cref{rmk:classical_bessel_specialization}.
Here we record the finite resolvent flow identities,
and formulate the separate finite-dimensional closure problem.

\subsection{Universal resolvent matrix and polynomial tilts}
\label{sub:universal_resolvent}

For an arbitrary symbol $\varphi$ satisfying the assumptions of \Cref{sub:hardy_setup}, define
\begin{equation*}
R_{m}\colon H\to\mathbb C^{m},\quad e_{p}^{\top}R_{m}=e_{p}^{\top}T(\varphi_{+}),\qquad C_{n}\colon\mathbb C^{n}\to H,\quad C_{n}e_{q}=T(\varphi_{-})e_{q},
\end{equation*}
where $m,n$ are arbitrary positive integers, $0\le p<m$, $0\le q<n$.
The pair $(R_{m},C_{n})$ is determined by the symbol $\varphi$, and 
does not depend on the tilt.

\begin{definition}\label{def:Y_universal}
The \emph{universal resolvent matrix} of $(R_{m},C_{n})$ at the symbol $\varphi$ is
\begin{equation*}
Y_{\varphi}^{m,n}\coloneqq R_{m}\,A^{-1}\,C_{n}\in\operatorname{Mat}_{m,n}(\mathbb C),\qquad A=I-K^{\varphi}.
\end{equation*}
\end{definition}

Let $A_{\xi}\in\mathbb C^{n\times N}$ and $A_{\theta}\in\mathbb C^{N\times m}$,
and define
\begin{equation}
\label{eq:chart_factorization}
R=A_{\theta}\,R_{m}\colon H\to\mathbb C^{N},\qquad C=C_{n}\,A_{\xi}\colon\mathbb C^{N}\to H.
\end{equation}
Then, by associativity,
\begin{equation}
\label{eq:Y_universal_main}
\Det_{\mathbb C^{N}}\bigl(R\,A^{-1}\,C\bigr)=\Det_{\mathbb C^{N}}\bigl(A_{\theta}\,Y_{\varphi}^{m,n}\,A_{\xi}\bigr).
\end{equation}

Let us now specialize to polynomial tilts.

\begin{proposition}\label{prop:Y_universal_tilted}
Let $\xi_{j}(z)$ be a polynomial of degree at most $d_{\xi}\le n-N$ in $z$
and $\theta_{i}(z)$ be a polynomial of degree at most $d_{\theta}\le m-N$ in
$z^{-1}$ for all $i,j$. Then the chart
$R=\Theta\,T(\varphi_{+})$,
$C=T(\varphi_{-})\,\Xi\,P_{N}$ of
\eqref{eq:tilted_chart} has the factorization 
\eqref{eq:chart_factorization} with banded matrices
\begin{equation}
\label{eq:A_banded}
(A_{\theta})_{i,p}=\widehat{\theta_{i}}(i-1-p),
\qquad
(A_{\xi})_{q,j}=\widehat{\xi_{j}}(q-j+1),
\end{equation}
where $1\le i,j\le N$, $0\le p<m$, $0\le q<n$.
Then
\begin{equation}
\label{eq:Y_universal_tilted}
G(\varphi)^{-N}\,D_{N}^{\xi,\theta}(\varphi)=\Det_{\mathbb C^{N}}\bigl(A_{\theta}\,Y_{\varphi}^{m,n}\,A_{\xi}\bigr).
\end{equation}
\end{proposition}
\begin{proof}
Expand the row tilt operator of \eqref{eq:Theta_def}:
$(\Theta h)_{i}=\sum_{a\ge 0}\widehat{\theta_{i}}(-a)\,h_{i-1+a}$. Since
$\theta_i$ has degree at most $d_\theta$ in $z^{-1}$, the $i$-th row of
$\Theta T(\varphi_+)$ is
\[
\sum_{a=0}^{d_\theta}\widehat{\theta_i}(-a)\,e_{i-1+a}^{\top}T(\varphi_+).
\]
The coefficient formula for $A_\theta$ in \eqref{eq:A_banded} therefore gives
$\Theta T(\varphi_+)=A_\theta R_m$; the degree bound $d_\theta\le m-N$ ensures
that all rows used by this sum lie among $0,\ldots,m-1$. Similarly, from
\eqref{eq:Xi_def}, the formula for $A_\xi$ gives
$T(\varphi_-)\Xi P_N=C_nA_\xi$, and the bound $d_\xi\le n-N$ ensures that all
columns used lie among $0,\ldots,n-1$. Thus the tilted chart has the
factorization \eqref{eq:chart_factorization}. Applying \eqref{eq:Y_universal_main}
and then \Cref{lem:tilted_minor_LHS} gives \eqref{eq:Y_universal_tilted}.
\end{proof}

\begin{remark}
For \eqref{eq:Y_universal_main}, the matrices $A_{\xi},A_{\theta}$ can be arbitrary.
However, to connect them to a finite chart $(R,C)$ as in 
\Cref{sec:finite_rank_specialization}, the matrices
$A_\theta R_m$ and $C_n A_\xi$ must have rank $N$, and the Gram matrix
$A_{\theta}R_{m}C_{n}A_{\xi}$ must be invertible.
The polynomial tilts with bounded degrees provide one such family of examples.
\end{remark}

\begin{definition}\label{def:fredholm_tau}
For a symbol $\varphi$ and chart maps $R,C$ (fixed, or built from $\varphi$ as
in \eqref{eq:tilted_chart}) with $\Gamma\coloneqq RC$ invertible on
$\mathbb C^{N}$, the \emph{oblique Fredholm tau} is
\begin{equation*}
\mathcal T_{R,C}(\varphi)\coloneqq \Det_{\mathbb C^{N}}(\Gamma)\,\Det_{\operatorname{Ker}(R)}\bigl(I_{\operatorname{Ker}(R)}-\Pi_{V}K|_{\operatorname{Ker}(R)}\bigr).
\end{equation*}
\end{definition}

When $R,C$ depend on parameters, this notation means that $R,C,K$ and
$\Pi_{V}$ are evaluated at the same parameter value. Equivalently, by \Cref{thm:finite_GBO},
\begin{equation}
\label{eq:fredholm_tau_jacobi}
\mathcal T_{R,C}(\varphi)=Z^{-1}\,\Det_{\mathbb C^{N}}\bigl(R\,A^{-1}\,C\bigr)=\Det_{H}(I-K)\,\Det_{\mathbb C^{N}}\bigl(R\,A^{-1}\,C\bigr),
\end{equation}
so identity \eqref{eq:tilted_finite_GBO} reads $D_{N}^{\xi,\theta}(\varphi)=G(\varphi)^{N}\,Z\,\mathcal T_{R,C}(\varphi)$.

\subsection{Symmetric finite Laurent symbols}
\label{sub:finite_laurent}

Fix $M\ge 1$ and $\mathbf t=(t_{1},\ldots,t_{M})\in\mathbb R^{M}$; the
identities below extend holomorphically to $\mathbf t\in\mathbb C^{M}$. The
symmetric finite Laurent exponential symbol 
$\varphi(z;\mathbf t)$
is defined by \eqref{eq:Laurent_symbol_finite_times}.
Its Wiener--Hopf factors are $\varphi_{\pm}(z)=\exp\bigl(\sum_{r}t_{r}z^{\pm r}\bigr)$,
thus, 
$G(\varphi)=1$, and the strong Szeg\H o constant is $Z_{\mathbf t}=\exp\bigl(\sum_{r=1}^{M}r\,t_{r}^{2}\bigr)$.
This implies that
\begin{equation*}
\Det_{H}(I-K_{\mathbf t})=e^{-\sum_{r}r\,t_{r}^{2}},\qquad
\partial_{t_{r}}\log\Det_{H}(I-K_{\mathbf t})=-2r\,t_{r}.
\end{equation*}
Let $S,S^{*}$ denote the right and left shifts on $H=\ell^{2}(\mathbb Z_{\ge 0})$, $S e_{k}=e_{k+1}$ and $S^{*}e_{0}=0$, $S^{*}e_{k}=e_{k-1}$ for $k\ge1$.
The ratio
\begin{equation*}
b(z;\mathbf t)\coloneqq \frac{\varphi_{-}(z)}{\varphi_{+}(z)}=\exp\Bigl(\sum_{r}t_{r}(z^{-r}-z^{r})\Bigr)
\end{equation*}
has Fourier coefficients $b_{n}(\mathbf t)$. Since $\widetilde c=b$ in this symmetric case,
the kernel has the form
$K_{\mathbf t}=\mathcal H(b)\mathcal H(b)^{\top}$.

\begin{remark}
By \cite[Remark~2]{BorodinOkounkov2000}, for
$i\neq j$, we can write the BOGC kernel as follows:
\[
K_{\mathbf t}(i,j)
=-\frac{1}{i-j}\sum_{r=1}^{M}r t_r\sum_{a=0}^{r-1}
\left(b_{i-a}(\mathbf t)b_{j+r-a}(\mathbf t)-b_{j-a}(\mathbf t)b_{i+r-a}(\mathbf t)\right).
\]
By \Cref{lem:finite_rank_oblique_correction},
the tilted kernel corresponds to a rank $N$ correction of $K_{\mathbf t}$, that is,
\[
(\Pi_{V}K_{\mathbf t})(i,j)
=-\frac{1}{i-j}\sum_{r=1}^{M}r t_r\sum_{a=0}^{r-1}
\left(b_{i-a}(\mathbf t)b_{j+r-a}(\mathbf t)-b_{j-a}(\mathbf t)b_{i+r-a}(\mathbf t)\right)-\sum_{\alpha=1}^{N}c_{\alpha}(i)\,\psi_{\alpha}(j),
\]
where $c_{\alpha}=Cf_{\alpha}\in H$ and
$(\psi_{\alpha}(j))_{j\ge0}=f_{\alpha}^{\top}\Gamma^{-1}R K_{\mathbf t}e_j$.
\end{remark}

For $1\le r\le M$ and $0\le a\le r-1$, define
\begin{equation*}
\rho_{a}^{(r)}[\ell]\coloneqq b_{\ell+a+1-r}(\mathbf t),\qquad
h_{a}^{(r)}\coloneqq \mathcal H(b)\,\rho_{a}^{(r)}.
\end{equation*}

\begin{lemma}\label{lem:finite_laurent_regularity}
In the above setting, the
coefficients
$b_n(\mathbf t)$ decay faster than exponentially as $n\to\pm\infty$, locally
uniformly for $\mathbf t$ in compact subsets of $\mathbb C^M$. Consequently
$\mathcal H(b)$ and $\partial_{t_r}\mathcal H(b)$ are Hilbert--Schmidt,
$K_{\mathbf t}=\mathcal H(b)\mathcal H(b)^{\top}$ and $\partial_{t_{r}}K_{\mathbf t}$ are trace class,
and $\rho_a^{(r)},h_a^{(r)}\in H$.
\end{lemma}

\begin{proof}
Cauchy's formula on $|z|=R>1$
gives, for $n\ge0$,
\[
|b_n(\mathbf t)|\le
\exp\bigl(C_{\mathbf t}(R^M+R^{-M})\bigr)R^{-n},
\]
and optimizing in $R$ gives superexponential decay. The coefficients with
$n<0$ are handled by applying the same estimate on $|z|=R^{-1}$. The same
bounds apply to $\partial_{t_r}b=(z^{-r}-z^r)b$. Its Fourier coefficients are finite
linear combinations of $b_{n-r}$ and $b_{n+r}$, so they also decay faster than
exponentially. Therefore $\partial_{t_r}\mathcal H(b)=\mathcal H(\partial_{t_r}b)$
is Hilbert--Schmidt by the same criterion; indeed,
\[
\sum_{i,\ell\ge0}|b_{i+\ell+1}|^2=\sum_{k\ge1}k|b_k|^2<\infty,
\]
and the same calculation for $\partial_{t_r}b$ applies. Products of two Hilbert--Schmidt operators are trace class, giving the assertion for $K_{\mathbf t}$ and $\partial_{t_{r}}K_{\mathbf t}$ (the latter by the Leibniz rule). Finally,
$\rho_a^{(r)}$ is a shift of the coefficient sequence of $b$, and
$h_a^{(r)}=\mathcal H(b)\rho_a^{(r)}$.
\end{proof}

\begin{proposition}\label{prop:tr_flow_HK}
For each $1\le r\le M$, we have
\begin{align}
\label{eq:H_prime_finite_laurent}
\partial_{t_{r}}\mathcal H(b) &= (S^{*})^{r}\mathcal H(b)-S^{r}\,\mathcal H(b)-\sum_{a=0}^{r-1}e_{a}\,(\rho_{a}^{(r)})^{\top},\\
\label{eq:K_prime_finite_laurent}
\partial_{t_{r}}K_{\mathbf t} &= (S^{*})^{r}K_{\mathbf t}-S^{r}\,K_{\mathbf t}+K_{\mathbf t}\,S^{r}-K_{\mathbf t}\,(S^{*})^{r}-\sum_{a=0}^{r-1}\bigl[e_{a}\,(h_{a}^{(r)})^{\top}+h_{a}^{(r)}\,e_{a}^{\top}\bigr].
\end{align}
The derivatives above are taken in the Hilbert--Schmidt
class for the Hankel operators and in the trace-class sense for $K_{\mathbf t}$.
\end{proposition}
\begin{proof}
For \eqref{eq:H_prime_finite_laurent}, $\partial_{t_{r}}b(z)=(z^{-r}-z^{r})\,b(z)$, so $\partial_{t_{r}}b_{n}=b_{n+r}-b_{n-r}$. Hence $\partial_{t_{r}}\mathcal H(b)_{i,\ell}=b_{i+\ell+1+r}-b_{i+\ell+1-r}$. The identities $((S^{*})^{r}\mathcal H(b))_{i,\ell}=b_{i+\ell+1+r}$ and $(S^{r}\mathcal H(b))_{i,\ell}=b_{i+\ell+1-r}$ hold for $i\ge r$. For $0\le i<r$, the second shifted term is zero, and the boundary sum subtracts $b_{\ell+a+1-r}$ in row $i=a$.

For \eqref{eq:K_prime_finite_laurent}, differentiate $K_{\mathbf t}=\mathcal H(b)\,\mathcal H(b)^{\top}$ by the Leibniz rule:
\[
\partial_{t_{r}}K_{\mathbf t}=(\partial_{t_{r}}\mathcal H(b))\,\mathcal H(b)^{\top}+\mathcal H(b)\,(\partial_{t_{r}}\mathcal H(b))^{\top}.
\]
Since $((S^{*})^{r})^{\top}=S^{r}$ and $(S^{r})^{\top}=(S^{*})^{r}$, the transpose of the first two terms in \eqref{eq:H_prime_finite_laurent} contributes $K_{\mathbf t} S^{r}-K_{\mathbf t}(S^{*})^{r}$.  The transpose of the boundary term contributes $-\sum_{a}h_{a}^{(r)}e_{a}^{\top}$, while the non-transposed boundary term contributes $-\sum_{a}e_{a}(h_{a}^{(r)})^{\top}$.
\end{proof}

\begin{proposition}\label{prop:tr_flow_Y}
Fix the tilts $\xi_j,\theta_i$ as in \Cref{sec:tilted_toeplitz_minors}, and let $\Xi,\Theta$ be the
corresponding tilt operators (\Cref{def:tilt_operators}).  Set
$R(\mathbf t)=\Theta T(\varphi_{+}(\mathbf t))$ and
$C(\mathbf t)=T(\varphi_{-}(\mathbf t))\Xi P_{N}$.
For
each $1\le r\le M$, set $Q(\mathbf t)\coloneqq(I-K_{\mathbf t})^{-1}$ and
$Y(\mathbf t)\coloneqq R(\mathbf t)Q(\mathbf t)C(\mathbf t)$. Suppressing
$\mathbf t$ from the notation, we have
\begin{equation}
\label{eq:Y_prime_finite_laurent}
\partial_{t_{r}}Y=R\,(S^{*})^{r}\,Q\,C+R\,Q\,S^{r}\,C-\sum_{a=0}^{r-1}\bigl[(R\,Q\,e_{a})((h_{a}^{(r)})^{\top}\,Q\,C)+(R\,Q\,h_{a}^{(r)})(e_{a}^{\top}\,Q\,C)\bigr].
\end{equation}
\end{proposition}

\begin{proof}
Since $T(\varphi_{+})=\exp(\sum_{s}t_{s}\,S^{s})$ and $T(\varphi_{-})=\exp(\sum_{s}t_{s}\,(S^{*})^{s})$, we have $\partial_{t_{r}}R=R\,S^{r}$ and $\partial_{t_{r}}C=(S^{*})^{r}\,C$. Differentiating $Y=RQC$ and using $\partial_{t_{r}}Q=Q\,(\partial_{t_{r}}K)\,Q$ gives
\begin{equation*}
\partial_{t_{r}}Y=R\,S^{r}Q C+R Q(S^{*})^{r}C+RQ(\partial_{t_{r}}K)QC.
\end{equation*}
Substitute \eqref{eq:K_prime_finite_laurent}. The four pure-shift products telescope as follows, using $KQ=QK=Q-I$:
\begin{align*}
RQ(S^{*})^{r}KQC&=RQ(S^{*})^{r}QC-RQ(S^{*})^{r}C,\\
RQS^{r}KQC&=RQS^{r}QC-RQS^{r}C,\\
RQKS^{r}QC&=RQS^{r}QC-RS^{r}QC,\\
RQK(S^{*})^{r}QC&=RQ(S^{*})^{r}QC-R(S^{*})^{r}QC.
\end{align*}
With the signs $+,-,+,-$ from \eqref{eq:K_prime_finite_laurent}, the $RQ(S^{*})^{r}QC$ and $RQS^{r}QC$ terms cancel pairwise. The remaining pure-shift terms are cancelled by $R S^{r}QC+RQ(S^{*})^{r}C$, leaving $R(S^{*})^{r}QC+RQ S^{r}C$. The boundary part of \eqref{eq:K_prime_finite_laurent} gives the finite rank sum in \eqref{eq:Y_prime_finite_laurent}.
\end{proof}

\begin{corollary}\label{cor:rectangular_Y_flow}
For the time-dependent universal maps $R_m(\mathbf t),C_n(\mathbf t)$, the same
identity holds for the rectangular universal block
$Y_{\varphi}^{m,n}=R_mQC_n$:
\begin{equation}
\label{eq:Y_rect_prime_finite_laurent}
\partial_{t_r}Y_{\varphi}^{m,n}
=R_m(S^*)^rQC_n+R_mQS^rC_n-
\sum_{a=0}^{r-1}
\bigl[(R_mQe_a)((h_a^{(r)})^\top QC_n)+(R_mQh_a^{(r)})(e_a^\top QC_n)\bigr].
\end{equation}
Consequently, for polynomial tilts and on the open set where
$A_{\theta}Y_{\varphi}^{m,n}A_{\xi}$ is invertible, we have
\begin{equation*}
\partial_{t_r}\log\Det_{\mathbb C^N}\bigl(A_{\theta}Y_{\varphi}^{m,n}A_{\xi}\bigr)
=\operatorname{tr}_{\mathbb C^N}\left[
\bigl(A_{\theta}Y_{\varphi}^{m,n}A_{\xi}\bigr)^{-1}
A_{\theta}(\partial_{t_r}Y_{\varphi}^{m,n})A_{\xi}
\right].
\end{equation*}
\end{corollary}

\begin{proof}
The proof of \Cref{prop:tr_flow_Y} applies verbatim to $R_m(\mathbf t),C_n(\mathbf t)$,
since $\partial_{t_r}R_m=R_mS^r$ and $\partial_{t_r}C_n=(S^*)^rC_n$. The logarithmic
derivative of a determinant is a standard fact for finite-dimensional matrices.
\end{proof}

\begin{corollary}\label{cor:tr_log_derivative}
Fix the tilts and set
\[
\Gamma(\mathbf t)=R(\mathbf t)C(\mathbf t),\qquad
Y(\mathbf t)=R(\mathbf t)Q(\mathbf t)C(\mathbf t),
\qquad
\mathcal T(\mathbf t)\coloneqq
\mathcal T_{R(\mathbf t),C(\mathbf t)}(\varphi(\cdot;\mathbf t)).
\]
On any open set where $\Gamma(\mathbf t)$ and $Y(\mathbf t)$ are invertible, for
$1\le r\le M$ we have
\begin{equation*}
\partial_{t_{r}}\log \mathcal T(\mathbf t)=-2r\,t_{r}+\operatorname{tr}_{\mathbb C^{N}}\bigl(Y(\mathbf t)^{-1}\,\partial_{t_{r}}Y(\mathbf t)\bigr).
\end{equation*}
\end{corollary}

\begin{proof}
By \eqref{eq:fredholm_tau_jacobi},
$\log \mathcal T(\mathbf t)=\log\Det_{H}(I-K_{\mathbf t})+
\log\Det_{\mathbb C^{N}}(Y(\mathbf t))$ on this open set. The first summand is
$-\sum_{s}s\,t_{s}^{2}$, whose $t_{r}$-derivative is $-2r\,t_{r}$. The second
yields $\operatorname{tr}(Y(\mathbf t)^{-1}\,\partial_{t_{r}}Y(\mathbf t))$ 
in the same way as in \Cref{cor:rectangular_Y_flow}.
\end{proof}

\begin{remark}\label{rmk:classical_bessel_specialization}
For $M=1$, \eqref{eq:Laurent_symbol_finite_times} is the Bessel symbol
$\varphi_{t}(z)=e^{t(z+z^{-1})}$. The pure-shift tilts produce determinants of the form
\[
\Det\bigl[I_{\nu+i-j}(2t)\bigr]_{i,j=0}^{n-1}.
\]
Up to elementary normalizations
and indexing conventions, such determinants are known to form
Painlev\'e~III$'$ tau-function sequences; see Forrester--Witte
\cite{ForresterWitte2004} and, for broader tau-function background,
\cite{ForresterWitte2002}. For recent asymptotic work, see
\cite{ChenXuZhao2024}. The associated OPUC reflection coefficients  
and related finite Laurent Toeplitz determinants satisfy discrete Painlev\'e-type
recurrences; see also \cite{ChouteauTarricone2023}.
\end{remark}

\subsection{Finite-dimensional closure problem}
\label{sub:schlesinger}

For finite Laurent exponential symbols and polynomial tilts of bandwidth at most
$d$, \Cref{sub:universal_resolvent,sub:finite_laurent} give the first
$t_r$-derivatives of the finite blocks in terms of the following finite list of
resolvent matrix elements:
\begin{equation}
\label{eq:resolvent_variables}
\begin{gathered}
Y_{\varphi}^{m,n},\quad R_m(S^*)^rQC_n,\quad R_mQS^rC_n,\\
R_{m}\,Q\,e_{a},\quad R_{m}\,Q\,h_{a}^{(r)},\quad e_{a}^{\top}\,Q\,C_{n},\quad
(h_{a}^{(r)})^{\top}\,Q\,C_{n},
\end{gathered}
\qquad (1\le r\le M,\,0\le a\le r-1).
\end{equation}
The shifted blocks $R_m(S^*)^rQC_n$ and $R_mQS^rC_n$ are part of the data in
\eqref{eq:Y_rect_prime_finite_laurent}; they are not, in general, entries of
the original block $Y_{\varphi}^{m,n}$. 
Thus, \eqref{eq:resolvent_variables} is 
not yet a closed finite-dimensional
system.

\begin{conjecture}\label{conj:schlesinger_closure}
Assume that $\varphi$ admits a nonvanishing analytic extension to an annulus
containing $\mathbb T$, has zero winding number on $\mathbb T$, and, for one
branch of $\log\varphi$ on this annulus, the logarithmic derivative extends to
a rational function,
\begin{equation*}
z\,\partial_{z}\log\varphi(z)\in\mathbb C(z).
\end{equation*}
Assume also that the tilts $(\xi_j,\theta_i)$ are finite polynomial or rational
functions, with the corresponding chart nondegenerate on the parameter domain under
consideration. Then, after adjoining finitely many auxiliary shifted-boundary
or residue variables to the variables in
\eqref{eq:resolvent_variables}, the associated deformation equations for
$\mathcal T_{R,C}$ should admit a finite-dimensional IIKS/Riemann--Hilbert
realization. In such a realization $\mathcal T_{R,C}$, up to explicit scalar
factors, should coincide with the corresponding isomonodromic tau function.
\end{conjecture}

\begin{remark}\label{rmk:closure_not_linear}
A finite closure, if present, cannot simply come from
minors of the original block alone. At a minimum, it should also use
Pl\"ucker relations among minors of enlarged blocks, as well as rational
functions of the deformation parameters.

Numerical experiments point in the same direction.  In the test case of
polynomial tilts with matrix size $N=3$ and degree bound $d=2$, differentiating
the minors of the original finite block produced quantities which could not be
written as constant linear combinations of the same minors.  After adding the
nearest boundary-shifted minors, the number of numerically independent sampled
quantities increased from $12$ to $15$, that is, by three. A separate check on
the corresponding resolvent quantities showed the same numerical increase,
again from $12$ to $15$.  Thus the expected finite closure, if it exists, should
allow coefficients depending on the times, at least rationally, rather than only
a fixed finite-dimensional span with constant coefficients.
\end{remark}

\appendix

\section{On spiked soft edge asymptotics of tilted Fredholm formulas}
\label[appendix]{sec:appendix_spiked}

Let us illustrate how the tilted Fredholm
determinant on the canonical tail space
$Q_NH=\overline{\operatorname{span}}\{e_N,e_{N+1},\ldots\}$
(\Cref{prop:fixed_tail_tilted}) can be used in soft edge asymptotic
computations.  
We work out a rank one ``spiked'' example in the sense of
Baik--Ben Arous--P\'ech\'e \cite{BBP2005phase}, 
presenting only asymptotic computations without rigorous convergence estimates.

\subsection{}

Consider the finite Laurent symbol
\begin{equation*}
    \varphi_L(z)=\exp L\bigl(a(z+z^{-1})+b(z^2+z^{-2})\bigr),
    \qquad a>0,\quad 0<b<a/8 .
\end{equation*}
Then $G(\varphi_L)=1$ and the strong Szeg\H{o} constant is
\begin{equation*}
    Z_L=\exp L^2(a^2+2b^2).
\end{equation*}
For the ordinary Toeplitz determinant, the BOGC identity gives
\begin{equation*}
    \frac{D_N(\varphi_L)}{Z_L}
    =\Det_{Q_NH}\bigl(I_{Q_NH}-Q_NK_LQ_N\bigr).
\end{equation*}
Set
\begin{equation*}
    \chi=2a-4b,\qquad c=(a-8b)^{1/3},\qquad
    N_L(s)=\bigl\lfloor \chi L+cL^{1/3}s\bigr\rfloor .
\end{equation*}
To see the soft edge scale, look at the Fourier coefficients of
\begin{equation}\label{eq:spiked_symbol_ratio}
    b_L(z)=\frac{\varphi_{L,-}(z)}{\varphi_{L,+}(z)}
    =\exp L\bigl(a(z^{-1}-z)+b(z^{-2}-z^2)\bigr).
\end{equation}
They are given by
\begin{equation*}
    b_{L,n}=\oint b_L(z)z^{-n}\frac{dz}{2\pi iz}.
\end{equation*}
The saddle point corresponding to $n\approx\chi L$ is $z=-1$.  Writing
$z=-e^\lambda$, and ignoring the harmless sign $(-1)^n$, the exponent in
\eqref{eq:spiked_symbol_ratio} with the Fourier factor $z^{-n}$ becomes
\begin{equation*}
    L\bigl(2a\sinh\lambda-2b\sinh(2\lambda)\bigr)-n\lambda .
\end{equation*}
At $n=\chi L$ the linear term cancels, and
\begin{equation*}
    2a\sinh\lambda-2b\sinh(2\lambda)-\chi\lambda
    =\frac{a-8b}{3}\lambda^3+O(\lambda^5).
\end{equation*}
Thus the Airy scale is $N=\chi L+O(L^{1/3})$.  At the level of Fourier
coefficients, the standard steepest descent at $z=-1$ gives
\begin{equation*}
    (-1)^n cL^{1/3}b_{L,n}\longrightarrow \operatorname{Ai}(x),
    \qquad n=\chi L+cL^{1/3}x+O(1).
\end{equation*}
Consequently, after conjugation by $(-1)^i$ and the usual Riemann-sum
identification of $Q_{N_L(s)}H$ with $L^2(s,\infty)$, one expects
$Q_{N_L(s)}K_LQ_{N_L(s)}$ to converge to the Airy kernel
$K_{\operatorname{Ai}}$ on $L^2(s,\infty)$.

\subsection{}

Consider now the tilted situation, with the tilts depending on $L$. Assume that the hypotheses of
\Cref{prop:fixed_tail_tilted} hold for $N=N_L(s)$.  Let
$\mathsf K_L^{\xi,\theta}$ be the fixed-tail kernel on $Q_NH$ from
\Cref{prop:fixed_tail_tilted}.  By \eqref{eq:fixed_tail_kernel_expansion},
it is equal to the BOGC kernel $Q_NK_LQ_N$ plus a finite-rank correction.

After the same sign conjugation and Riemann-sum identification as above,
suppose that
\begin{equation*}
    Q_NK_LQ_N\longrightarrow K_{\operatorname{Ai}},
    \qquad
    \mathsf K_L^{\xi,\theta}-Q_NK_LQ_N\longrightarrow F_\infty
\end{equation*}
in trace norm on $L^2(s,\infty)$, where $F_\infty$ is finite rank and
$\Gamma_L$ is the corresponding Gram matrix.  Then continuity of the Fredholm
determinant gives
\begin{equation*}
    \frac{D_N^{\xi,\theta}(\varphi_L)}
    {Z_L\,\det_{\mathbb C^N}(\Gamma_L)}
    \longrightarrow
    \Det_{L^2(s,\infty)}\bigl(I-K_{\operatorname{Ai}}-F_\infty\bigr).
\end{equation*}
The right-hand side is the Fredholm determinant (on $(s,\infty)$)
of the Airy kernel modified by a finite-rank perturbation,
which is a typical form for soft edge limits of spiked models
\cite{BBP2005phase}, \cite{BorodinPeche2009}.

\subsection{}

We now specialize to a single column tilt.\footnote{The
case of an arbitrary finite number $k$ of columns can be handled similarly, 
which would lead to the general $k$-spiked 
Baik--Ben Arous--P\'ech\'e (BBP) distribution
defined in \cite{BBP2005phase}.}
Fix $w>0$ and set
\begin{equation*}
    \alpha_L=-\exp\{-w/(cL^{1/3})\},
    \qquad |\alpha_L|<1.
\end{equation*}
Take
\begin{equation*}
    \xi_j(z)=1\quad (1\le j<N),
    \qquad
    \xi_N(z)=(1-\alpha_Lz)^{-1},
    \qquad
    \theta_i\equiv1 .
\end{equation*}
Set $C_L=T(\varphi_{L,-})\Xi_LP_N$.  In this column-only case
$R=P_NT(\varphi_{L,+})$, so $RQ_N=0$ and
$\operatorname{Ker}(R)=Q_NH$.  The triangular factor
$P_NT(\varphi_{L,+})P_N$ cancels from the Gram matrix, and the fixed-tail
kernel becomes
\begin{equation*}
    \mathsf K_L^{\xi,\mathbf 1}
    =Q_NK_LQ_N-Q_NC_L(P_NC_L)^{-1}P_NK_LQ_N,
\end{equation*}
provided $P_NC_L$ is invertible.  We now compute this correction explicitly.

Write
\begin{equation*}
    \varphi_{L,-}(z)=\sum_{m\ge0}h_mz^{-m},\qquad h_0=1.
\end{equation*}
For $0\le j\le N-2$,
\begin{equation*}
    C_Le_j=T(\varphi_{L,-})e_j=\sum_{m=0}^{j}h_m e_{j-m},
\end{equation*}
so these columns have no $Q_N$-tail.  For the last column,
\begin{equation*}
    \Xi_Le_{N-1}=\sum_{n\ge0}\alpha_L^n e_{N-1+n},
\end{equation*}
and hence
\begin{equation}\label{eq:spiked_column_coefficients}
    (C_Le_{N-1})_i
    =\sum_{\substack{m\ge0\\ i+m\ge N-1}}
    h_m\alpha_L^{i+m-N+1}.
\end{equation}
For $i=N+r$, $r\ge0$, the lower constraint on $m$ is automatic, so
\begin{equation*}
    Q_NC_Le_{N-1}
    =\left(\sum_{m\ge0}h_m\alpha_L^m\right)
    \sum_{r\ge0}\alpha_L^{r+1}e_{N+r}.
\end{equation*}
At the boundary row $i=N-1$,
\begin{equation*}
    (C_Le_{N-1})_{N-1}=\sum_{m\ge0}h_m\alpha_L^m.
\end{equation*}
The first $N-1$ columns have no $e_{N-1}$ component.  Therefore $P_NC_L$ is
block upper triangular with a unit upper-triangular upper-left block and
lower-right entry $\sum_{m\ge0}h_m\alpha_L^m$.  In particular, $P_NC_L$ is
invertible, and
\begin{equation*}
    Q_NC_L(P_NC_L)^{-1}
    =
    \left(\sum_{r\ge0}\alpha_L^{r+1}e_{N+r}\right)e_{N-1}^{*}.
\end{equation*}
Consequently,
\begin{equation}\label{eq:spiked_column_kernel_vector_form}
    \mathsf K_L^{\xi,\mathbf 1}
    =Q_NK_LQ_N
    -\left(\sum_{r\ge0}\alpha_L^{r+1}e_{N+r}\right)
    \otimes(e_{N-1}^{*}K_LQ_N).
\end{equation}
Equivalently, for $i,j\ge N$,
\begin{equation}\label{eq:spiked_column_kernel}
    \mathsf K_L^{\xi,\mathbf 1}(i,j)
    =K_L(i,j)-\alpha_L^{i-N+1}K_L(N-1,j).
\end{equation}
Thus the one-column tilt gives an exact rank-one correction of the ordinary
BOGC tail kernel.

\subsection{}

Put $N=N_L(s)$ and write
\begin{equation*}
    x=\frac{i-\chi L}{cL^{1/3}},\qquad
    y=\frac{j-\chi L}{cL^{1/3}}.
\end{equation*}
For $i=N+r$,
\begin{equation*}
    \alpha_L^{i-N+1}
    =(-1)^{r+1}\exp\{-w(r+1)/(cL^{1/3})\}.
\end{equation*}
After the same sign conjugation used for the ordinary Airy limit, this vector
factor tends to $\exp\{-w(x-s)\}$.

For the present symmetric symbol, $\widetilde c_L=b_L$, and therefore
\begin{equation*}
    K_L(i,j)=\sum_{\ell\ge0}b_{L,i+\ell+1}b_{L,j+\ell+1},
    \qquad
    K_L(N-1,j)=\sum_{\ell\ge0}b_{L,N+\ell}b_{L,j+\ell+1}.
\end{equation*}
Inserting the coefficient Airy asymptotics gives the Riemann-sum limit
\begin{equation*}
    cL^{1/3}(-1)^{N-1+j}K_L(N-1,j)
    \longrightarrow
    \int_0^\infty
    \operatorname{Ai}(s+t)\operatorname{Ai}(y+t)\,dt
    =K_{\operatorname{Ai}}(s,y).
\end{equation*}
Hence the rank-one term in \eqref{eq:spiked_column_kernel} has candidate
Airy-scale limit
\begin{equation*}
    cL^{1/3}(-1)^{i+j}\alpha_L^{i-N+1}K_L(N-1,j)
    \longrightarrow e^{-w(x-s)}K_{\operatorname{Ai}}(s,y).
\end{equation*}
The limiting fixed-tail kernel is therefore expected to be the boundary form
\begin{equation}
\label{eq:spiked_column_airy_limit_boundary_form}
    K_{\operatorname{Ai}}(x,y)-E_w(x)K_{\operatorname{Ai}}(s,y),
    \qquad
    E_w(x)=e^{-w(x-s)},\qquad x,y>s.
\end{equation}

The boundary form is not the standard one-spike contour kernel at
the operator level.  The next \Cref{prop:boundary_bbp_pushthrough}
shows that, nevertheless, its
Fredholm determinant 
on $(s,\infty)$ coincides with the one-spike BBP distribution.

First, we need some notation.
Fix $s\in\mathbb R$ and $w>0$.  Set
\begin{equation}\label{eq:Phi_w_contour}
    \Phi_w(x)=\int_0^\infty e^{-wt}\operatorname{Ai}(x+t)\,dt
    =\frac{1}{2\pi i}\int_{\Gamma}
    \frac{\exp\{Z^3/3-xZ\}}{Z-a}\,dZ,
    \qquad a=-w,
\end{equation}
where $\Gamma$ is an Airy contour passing to the right of $a$.  Let $\Sigma$
be the dual Airy contour chosen so that
\begin{equation*}
    \operatorname{Ai}(y)=\frac{1}{2\pi i}\int_{\Sigma}
    \exp\{-W^3/3+yW\}\,dW
\end{equation*}
and
\begin{equation*}
    K_{\operatorname{Ai}}(x,y)
    =\frac{1}{(2\pi i)^2}
    \int_{\Gamma}dZ\int_{\Sigma}dW\,
    \frac{\exp\{Z^3/3-xZ\}}{\exp\{W^3/3-yW\}}\,
    \frac{1}{Z-W}.
\end{equation*}
Finally, let
\begin{equation*}
    \mathcal A\colon L^2(0,\infty)\to L^2(s,\infty),
    \qquad
    (\mathcal Af)(x)=\int_0^\infty \operatorname{Ai}(x+t)f(t)\,dt .
\end{equation*}
Thus $K_{\operatorname{Ai}}=\mathcal A\mathcal A^*$ and
$K_{\operatorname{Ai}}(s,\cdot)=\mathcal A(\operatorname{Ai}(s+\cdot))$.
We write $f\otimes g$ for the rank-one kernel $f(x)g(y)$.

\begin{proposition}
\label{prop:boundary_bbp_pushthrough}
With the notation above,
\begin{equation}\label{eq:boundary_bbp_det_equivalence}
\Det_{L^2(s,\infty)}
\bigl(I-K_{\operatorname{Ai}}+E_w\otimes K_{\operatorname{Ai}}(s,\cdot)\bigr)
=
\Det_{L^2(s,\infty)}
\bigl(I-K_{\operatorname{Ai}}+\Phi_w\otimes\operatorname{Ai}\bigr),
\end{equation}
and
\begin{equation}\label{eq:one_pole_bbp_ratio_form}
\begin{aligned}
&K_{\operatorname{Ai}}(x,y)-\Phi_w(x)\operatorname{Ai}(y)
\\
&\qquad =
\frac{1}{(2\pi i)^2}
\int_{\Gamma} dZ\int_{\Sigma} dW\,
\frac{\exp\{Z^3/3-xZ\}}{\exp\{W^3/3-yW\}}\,
\frac{1}{Z-W}\,
\frac{W-a}{Z-a},
\end{aligned}
\end{equation}
which is the standard one-spike BBP kernel from
\cite{BBP2005phase}.
\end{proposition}

\begin{proof}
The factorization identities above give
\[
    K_{\operatorname{Ai}}-E_w\otimes K_{\operatorname{Ai}}(s,\cdot)
    =
    \bigl(\mathcal A-E_w\otimes\operatorname{Ai}(s+\cdot)\bigr)\mathcal A^* .
\]
By Sylvester's identity $\Det(I-BC)=\Det(I-CB)$,
\[
\begin{aligned}
&\Det_{L^2(s,\infty)}
\bigl(I-K_{\operatorname{Ai}}+E_w\otimes K_{\operatorname{Ai}}(s,\cdot)\bigr)
\\
&\qquad =
\Det_{L^2(0,\infty)}
\Bigl(I-\mathcal A^*\mathcal A
+\mathcal A^*E_w\otimes\operatorname{Ai}(s+\cdot)\Bigr).
\end{aligned}
\]
Moreover,
\[
    \mathcal A^*E_w(t)
    =\int_s^\infty \operatorname{Ai}(x+t)e^{-w(x-s)}\,dx
    =\int_0^\infty e^{-wr}\operatorname{Ai}(s+t+r)\,dr
    =\Phi_w(s+t).
\]
After the translation $L^2(0,\infty)\simeq L^2(s,\infty)$, this proves
\eqref{eq:boundary_bbp_det_equivalence}.

For \eqref{eq:one_pole_bbp_ratio_form}, use the elementary identity
\[
    \frac{1}{Z-W}\frac{W-a}{Z-a}
    =
    \frac{1}{Z-W}-\frac{1}{Z-a}.
\]
Substituting it into the double contour integral gives the Airy kernel minus
$\Phi_w(x)\operatorname{Ai}(y)$, using \eqref{eq:Phi_w_contour} and the dual
Airy representation above.
\end{proof}

Kernels similar to our initial expression \eqref{eq:spiked_column_airy_limit_boundary_form} 
appear in asymptotic analysis of the polynuclear growth model with external sources
\cite{baik2000limiting_BR_distribution}, \cite{imamura2004fluctuations}.

\newcommand{\etalchar}[1]{$^{#1}$}


\begin{thebibliography}{HJK{\etalchar{+}}24}

\bibitem[BBP05]{BBP2005phase}
J.~Baik, G.~{Ben Arous}, and S.~P\'ech\'e.
\newblock Phase transition of the largest eigenvalue for non-null complex
  sample covariance matrices.
\newblock {\em {Ann. Probab.}}, 33(5):1643--1697, 2005.
\newblock arXiv:math/0403022 [math.PR].

\bibitem[BD02]{BumpDiaconis2002}
D.~Bump and P.~Diaconis.
\newblock {{Toeplitz} minors}.
\newblock {\em J. Combin. Theory Ser. A}, 97(2):252--271, 2002.

\bibitem[BLL26]{BaikLiaoLiu2026}
J.~Baik, Y.~Liao, and Z.~Liu.
\newblock {Periodic {KPZ} fixed point with general initial conditions}.
\newblock {\em arXiv preprint}, 2026.
\newblock arXiv:2603.01964 [math.PR].

\bibitem[BO00]{BorodinOkounkov2000}
A.~Borodin and A.~Okounkov.
\newblock {A {Fredholm} determinant formula for {Toeplitz} determinants}.
\newblock {\em Integral Equations Operator Theory}, 37(4):386--396, 2000.
\newblock arXiv:math/9907165.

\bibitem[BO16]{borodin2016representations}
A.~Borodin and G.~Olshanski.
\newblock {\em {Representations of the Infinite Symmetric Group}}, volume 160
  of {\em Cambridge Studies in Advanced Mathematics}.
\newblock Cambridge University Press, 2016.

\bibitem[Bor98]{Borodin1998b}
A.~Borodin.
\newblock Biorthogonal ensembles.
\newblock {\em Nuclear Physics B}, 536:704--732, 1998.
\newblock arXiv:math/9804027 [math.CA].

\bibitem[B{\"o}t01]{Bottcher2001}
A.~B{\"o}ttcher.
\newblock {One more proof of the {Borodin}--{Okounkov} formula for {Toeplitz}
  determinants}.
\newblock {\em Integral Equations Operator Theory}, 41(1):123--125, 2001.
\newblock arXiv:math/0012200.

\bibitem[BP08]{BorodinPeche2009}
A.~Borodin and S.~Peche.
\newblock Airy kernel with two sets of parameters in directed percolation and
  random matrix theory.
\newblock {\em Jour. Stat. Phys.}, 132(2):275--290, 2008.
\newblock arXiv:0712.1086v3 [math-ph].

\bibitem[BR00]{baik2000limiting_BR_distribution}
J.~Baik and E.~Rains.
\newblock Limiting distributions for a polynuclear growth model with external
  sources.
\newblock {\em Jour. Stat. Phys.}, 100(3):523--541, 2000.
\newblock arXiv:math/0003130 [math.PR].

\bibitem[BS06]{BottcherSilbermann}
A.~B\"ottcher and B.~Silbermann.
\newblock {\em {Analysis of {T}oeplitz Operators}}.
\newblock Springer-Verlag, Berlin, second edition, 2006.

\bibitem[Buc02]{buch2002littlewood}
A.~S. Buch.
\newblock {A Littlewood--Richardson rule for the K-theory of Grassmannians}.
\newblock {\em Acta Math.}, 189(1):37--78, 2002.
\newblock arXiv:math/0004137 [math.AG].

\bibitem[BW00]{BasorWidom2000}
E.~L. Basor and H.~Widom.
\newblock {On a {Toeplitz} determinant identity of {Borodin} and {Okounkov}}.
\newblock {\em Integral Equations Operator Theory}, 37(4):397--401, 2000.
\newblock arXiv:math/9909010.

\bibitem[BW06]{BottcherWidom2006}
A.~B{\"o}ttcher and H.~Widom.
\newblock {Szeg{\H{o}} via {Jacobi}}.
\newblock {\em Linear Algebra Appl.}, 419(2--3):656--667, 2006.
\newblock arXiv:math/0604009.

\bibitem[CT23]{ChouteauTarricone2023}
T.~Chouteau and S.~Tarricone.
\newblock {Recursion relation for {Toeplitz} determinants and the discrete
  {Painlev\'e} {II} hierarchy}.
\newblock {\em SIGMA}, 19:030, 2023.
\newblock arXiv:2211.16898 [math-ph].

\bibitem[CW15]{CafassoWu2015}
M.~Cafasso and C.-Z. Wu.
\newblock {Tau functions and the limit of block {Toeplitz} determinants}.
\newblock {\em Int. Math. Res. Not.}, 2015(20):10339--10366, 2015.
\newblock arXiv:1404.5149 [math.AG].

\bibitem[CXZ24]{ChenXuZhao2024}
Y.~Chen, S.-X. Xu, and Y.-Q. Zhao.
\newblock {Asymptotics of the determinant of the modified {Bessel} functions
  and the second {Painlev\'e} equation}.
\newblock {\em Random Matrices Theory Appl.}, 13(1):2450003, 2024.
\newblock arXiv:2402.11233 [math-ph].

\bibitem[FK94]{fomin1994grothendieck}
S.~Fomin and A.N. Kirillov.
\newblock {Grothendieck polynomials and the Yang-Baxter equation}.
\newblock In {\em Proc. Formal Power Series and Alg. Comb}, pages 183--190,
  1994.

\bibitem[FW02]{ForresterWitte2002}
P.~J. Forrester and N.~S. Witte.
\newblock {Application of the {$\tau$}-function theory of {Painlev\'e}
  equations to random matrices: {$P_{\mathrm{V}}$}, {$P_{\mathrm{III}}$}, the
  {LUE}, {JUE} and {CUE}}.
\newblock {\em Comm. Pure Appl. Math.}, 55(6):679--727, 2002.
\newblock arXiv:math-ph/0201051.

\bibitem[FW04]{ForresterWitte2004}
P.~J. Forrester and N.~S. Witte.
\newblock {Discrete {Painlev\'e} equations, orthogonal polynomials on the unit
  circle, and {$N$}-recurrences for averages over {$U(N)$}:
  {$P_{\mathrm{III}'}$} and {$P_{\mathrm{V}}$} {$\tau$}-functions}.
\newblock {\em Int. Math. Res. Not.}, (4):160--183, 2004.
\newblock arXiv:math-ph/0305029.

\bibitem[GC79]{GeronimoCase1979}
J.~S. Geronimo and K.~M. Case.
\newblock {Scattering theory and polynomials orthogonal on the unit circle}.
\newblock {\em J. Math. Phys.}, 20:299--310, 1979.

\bibitem[Ges90]{gessel1990symmetric}
Ira~M. Gessel.
\newblock Symmetric functions and {P}-recursiveness.
\newblock {\em J. Combin. Theory Ser. A}, 53(2):257--285, 1990.

\bibitem[GGT20]{GarciaTierzGarcia2020}
D.~Garc{\'\i}a-Garc{\'\i}a and M.~Tierz.
\newblock {{Toeplitz} minors and specializations of skew {Schur} polynomials}.
\newblock {\em J. Combin. Theory Ser. A}, 172, 2020.
\newblock arXiv:1706.02574 [math.CO].

\bibitem[GP24]{GavrilovaPetrov2023_Groth}
S.~Gavrilova and L.~Petrov.
\newblock {Tilted biorthogonal ensembles, Grothendieck random partitions, and
  determinantal tests}.
\newblock {\em Selecta Math.}, 30:Article 56, 2024.
\newblock arXiv:2305.17747 [math.PR].

\bibitem[HJK{\etalchar{+}}24]{hwang2021refined}
B.-H. Hwang, J.~Jang, J.S. Kim, M.~Song, and U.-K. Song.
\newblock {Refined canonical stable Grothendieck polynomials and their duals,
  Part I}.
\newblock {\em Adv. Math.}, 446:109670, 2024.
\newblock arXiv:2104.04251 [math.CO].

\bibitem[IIKS90]{its1990differential}
A.R. Its, A.G. Izergin, V.E. Korepin, and N.A. Slavnov.
\newblock {Differential equations for quantum correlation functions}.
\newblock {\em Int. J. Mod. Phys. B}, 4(5):1003--1037, 1990.

\bibitem[IS04]{imamura2004fluctuations}
T.~Imamura and T.~Sasamoto.
\newblock Fluctuations of the one-dimensional polynuclear growth model with
  external sources.
\newblock {\em Nuclear Physics B}, 699(3):503--544, 2004.
\newblock arXiv:math-ph/0406001.

\bibitem[Koz14]{Kozlowski2013}
K.~K. Kozlowski.
\newblock {On lacunary {Toeplitz} determinants}.
\newblock {\em Asymptotic Analysis}, 88(1-2):1--16, 2014.
\newblock arXiv:1310.2584 [math-ph].

\bibitem[LS82]{lascoux1982structure}
A.~Lascoux and M.-P. Sch\"utzenberger.
\newblock {Structure de Hopf de l'anneau de cohomologie et de l'anneau de
  Grothendieck d'une vari\'et\'e de drapeaux}.
\newblock {\em C. R. Acad. Sci. Paris S\'er. I Math.}, 295(11):629--633, 1982.

\bibitem[LT26]{LiuTripathi2026}
Z.~Liu and T.~Tripathi.
\newblock {A determinant identity for the sum of contour integral matrices}.
\newblock {\em arXiv preprint}, 2026.
\newblock arXiv:2604.24747 [math.PR].

\bibitem[Mac95]{Macdonald1995}
I.G. Macdonald.
\newblock {\em Symmetric functions and {H}all polynomials}.
\newblock Oxford University Press, 2nd edition, 1995.

\bibitem[MMS17]{MaximenkoMoctezuma2017}
E.~A. Maximenko and M.~A. Moctezuma-Salazar.
\newblock {Cofactors and eigenvectors of banded {Toeplitz} matrices: {Trench}
  formulas via skew {Schur} polynomials}.
\newblock {\em Oper. Matrices}, 11(4), 2017.
\newblock arXiv:1705.08067 [math.CO].

\bibitem[MS13]{Motegi-Sakai13}
K.~Motegi and K.~Sakai.
\newblock {Vertex models, {TASEP} and {G}rothendieck polynomials}.
\newblock {\em J. Phys. A: Math. Theor.}, 46(35):355201, 2013.
\newblock arXiv:1305.3030 [math-ph].

\bibitem[Ons44]{Onsager1944Ising}
L.~Onsager.
\newblock {Crystal statistics. {I}. {A} two-dimensional model with an
  order-disorder transition}.
\newblock {\em Phys. Rev.}, 65:117--149, 1944.

\bibitem[Pel03]{Peller-Hankel}
V.~V. Peller.
\newblock {\em {Hankel Operators and Their Applications}}.
\newblock Springer Monographs in Mathematics. Springer-Verlag, New York, 2003.

\bibitem[Sim05a]{Simon-OPUC1}
B.~Simon.
\newblock {\em Orthogonal Polynomials on the Unit Circle, Part 1: Classical
  Theory}, volume~54 of {\em American Mathematical Society Colloquium
  Publications}.
\newblock American Mathematical Society, Providence, RI, 2005.

\bibitem[Sim05b]{Simon-trace-ideals}
B.~Simon.
\newblock {\em Trace ideals and their applications, second edition}, volume 120
  of {\em Mathematical Surveys and Monographs}.
\newblock AMS, 2005.

\bibitem[SW85]{SegalWilson1985}
G.~Segal and G.~Wilson.
\newblock {Loop groups and equations of {KdV} type}.
\newblock {\em Inst. Hautes \'Etudes Sci. Publ. Math.}, 61:5--65, 1985.

\bibitem[Sze15]{Szego1915}
G.~Szeg{\H{o}}.
\newblock {Ein {Grenzwertsatz} \"uber die {Toeplitz}schen {Determinanten} einer
  reellen positiven {Funktion}}.
\newblock {\em Math. Ann.}, 76:490--503, 1915.

\bibitem[Sze52]{Szego1952}
G.~Szeg{\H{o}}.
\newblock {On certain {Hermitian} forms associated with the {Fourier} series of
  a positive function}.
\newblock {\em Comm. S\'em. Math. Univ. Lund [Medd. Lunds Univ. Mat. Sem.]},
  pages 228--238, 1952.
\newblock Tome Suppl\'ementaire d\'edi\'e \`a Marcel Riesz.

\bibitem[TW02a]{TW_Airy_Painleve_2002}
C.~Tracy and H.~Widom.
\newblock Airy kernel and {P}ainlev{\'e} {II}.
\newblock In {\em Isomonodromic deformations and applications in physics
  ({M}ontr{\'e}al, {QC}, 2000)}, volume~31 of {\em CRM Proc. Lecture Notes},
  pages 85--96. AMS, 2002.
\newblock arXiv:solv-int/9901004.

\bibitem[TW02b]{TracyWidom2002}
C.~A. Tracy and H.~Widom.
\newblock {On the limit of some {Toeplitz}-like determinants}.
\newblock {\em SIAM J. Matrix Anal. Appl.}, 23(4):1194--1196, 2002.
\newblock arXiv:math/0107118.

\bibitem[Yel17]{yeliussizov2015duality}
D.~Yeliussizov.
\newblock {Duality and deformations of stable Grothendieck polynomials}.
\newblock {\em Jour. Alg. Comb.}, 45(1):295--344, 2017.
\newblock arXiv:1601.01581 [math.CO].

\end{thebibliography}
\end{document}